\documentclass[paper=a4, fontsize=11pt]{scrartcl}

\usepackage{amsmath,amsfonts,amsthm} 
\usepackage{bm} 
\usepackage{dsfont}
\usepackage[pdftex]{graphicx}	
\usepackage{enumitem}
\usepackage[numbers]{natbib}

\newtheorem*{theorem*}{Theorem}

\usepackage{url}
\usepackage{colortbl}
\usepackage{booktabs}
\usepackage{tabularx} 
\usepackage{float} 
\usepackage{multirow}%
\usepackage{fullpage}
\usepackage{natbib}
\usepackage{subcaption}  
\usepackage{graphicx}
\usepackage{tabularx} 
\usepackage{float} 
\usepackage{caption}

\usepackage{algorithm}
\usepackage{algpseudocode}
\usepackage{adjustbox}
\usepackage{xr-hyper} 

\usepackage[colorlinks=true,citecolor=blue,urlcolor=blue]{hyperref}
\usepackage{natbib}

\usepackage{graphicx}
\usepackage[english]{babel}

\usepackage{enumitem}
\usepackage{dsfont}
\usepackage{amsmath,amssymb}

\newcommand{\diam}{\operatorname{diam}}

\newcommand{\pr}{\mathbb{P}}

\newcommand{\argmin}{\operatornamewithlimits{\arg\min}}
\newcommand{\ind}{\mathds{1}}

\newcommand{\PP}{\mathbb{P}}

\newcommand{\EE}{\mathbb{E}}

\newtheorem{theorem}{Theorem}
\newtheorem{definition}[theorem]{Definition}

\newtheorem{corollary}[theorem]{Corollary}
\newtheorem{proposition}[theorem]{Proposition}
\newtheorem{lemma}[theorem]{Lemma}

\usepackage[T1]{fontenc}

\usepackage[english]{babel}															
\usepackage{amsmath,amsfonts,amsthm} 
\usepackage{bm} 
\usepackage{dsfont}
\usepackage[pdftex]{graphicx}	
\usepackage{enumitem}

\usepackage{url}
\usepackage{colortbl}
\usepackage{booktabs}
\usepackage{tabularx} 
\usepackage{float} 
\usepackage{multirow}%
\usepackage{fullpage}
\usepackage{natbib}
\usepackage{graphicx}
\usepackage{tabularx} 
\usepackage{float} 
\usepackage{caption}

\usepackage{xr-hyper} 

\usepackage[colorlinks=true,citecolor=blue,urlcolor=blue]{hyperref}
\usepackage{natbib}

\usepackage{chngcntr}
\counterwithout{table}{section}
\counterwithout{figure}{section}

\newcommand{\horrule}[1]{\rule{\linewidth}{#1}} 	

\title{
		\vspace{-1in} 	
		\usefont{OT1}{bch}{b}{n}
		\horrule{0.5pt} \\[0.4cm]
		\Large Pointwise convergence of purely random partition estimators: from random trees to prototype rules
		\\
		\horrule{2pt} \\[0.1cm]
}

\date{\large \today}

\author{\textbf{Jérémy Bettinger, François Portier, Adrien Saumard} \\
\large \href{mailto:jeremy.bettinger@ensai.fr}{jeremy.bettinger@ensai.fr} ; \href{mailto:francois.portier@ensai.fr}{francois.portier@ensai.fr} ; \href{mailto:adrien.saumard@ensai.fr}{adrien.saumard@ensai.fr} \\
\large Department of Statistics, \\
\large University of Rennes, ENSAI, CNRS, CREST-UMR 9194, F-35000 Rennes, France}

\begin{document}

\maketitle

\begin{abstract}
We study pointwise convergence rates of purely random partition estimators in nonparametric regression, where the partition -- into hyper-rectangles by purely random trees, or into Voronoi cells by prototype rules -- is built independently of the responses. Our analysis rests on a single geometric criterion, shape regularity, relating the diameter of a cell to its volume, which is shown in \citep{BPS} to be necessary and sufficient, up to logarithmic factors, for achieving the minimax rate $n^{-1/(d+2)}$. We show that centered and uniform trees are not shape-regular -- their cells' aspect ratio grows exponentially with the number of splits with probability bounded away from zero -- explaining the super-logarithmic corrections in their error bounds, whereas Mondrian trees, whose splits adapt to the current cell geometry, are shape-regular in probability and attain the minimax rate. The same analysis applied to Voronoi partitions yields the first pointwise concentration bounds for Proto-NN, resolving an open problem of \citep{gyorfi2021universal}, and shows that OptiNet achieves the minimax rate with markedly better success probability -- even almost surely, for a suitable choice of parameters -- thanks to its $\eta$-net construction.
\end{abstract}

\section{Introduction}\label{s1}

Partition-based estimators form a major class of methods for nonparametric regression. For a given point $x$, it results from a two-step procedure: (i) partitioning the covariate space, thereby assigning to each $x$ the unique element of the partition that contains it, and (ii) locally averaging the responses associated to the sample points falling in that element. Within this framework, a class of particular interest is that of \textit{purely random} partitions, built independently of the observed sample. Two families dominate the literature, distinguished only by the geometry of their cells: partitions into hyper-rectangles generated by purely random trees \citep{breiman2000some,arlot2014analysis,biau2016random} -- including their centered, uniform, and Mondrian variants \citep{lakshminarayanan2014mondrian} -- and Voronoi partitions generated by an auxiliary sample of prototypes, independent of the original data \citep{gyorfi2021universal,kerem2023error}, known as nearest neighbor-based prototype learning rules \citep[Chapter 19]{devroye96probabilistic}. Data-dependent partitions, such as CART \citep{breiman1984classification} or statistically equivalent blocks \citep{anderson}, do not fit the \textit{purely random} partitions framework; we refer to \cite{devroye96probabilistic} for an account of partition-based estimators.

A cornerstone result on partition-based estimators is Theorem 6.1 in \cite{devroye96probabilistic}: whenever the diameter of each cell vanishes while the number of points per cell grows, the resulting estimator is consistent (see Section \ref{s2} for a precise statement). Satisfactory as it is from the consistency viewpoint, this result is silent about convergence rates, and little is known about the rates of general partition-based estimators. The reason is that rates, unlike consistency, are governed by a quantitative balance between two features of the cell containing $x$: its diameter, which drives the bias, and its volume, which drives the number of points it captures and hence the variance. Achieving the minimax rate $n^{-1/(d+2)}$ for Lipschitz regression thus requires the diameter and the volume of the cell to decay at comparable rates -- a genuinely geometric requirement, formalized by the \textit{shape regularity} condition of \cite{BPS}, which is necessary and sufficient, up to logarithmic factors, for optimal pointwise and uniform estimation. This single criterion puts hyper-rectangular and Voronoi cells on the same footing, and reduces the rate question, for any purely random partition, to one and the same geometric question: does the randomness used to build the partition preserve well-shaped cells?

\noindent\textit{State of the art for purely random trees.} For tree-based partitions, the question above amounts to asking whether axis-aligned splits, performed blindly with respect to the current shape of the cell, keep the aspect ratio of the cells under control. A centered tree splits each cell at its midpoint along a direction selected uniformly at random; a uniform tree is similar, except that the split location is drawn uniformly within the cell. Most available results concern $L_2$ consistency: Mondrian trees attain the optimal rate of convergence for Lipschitz functions \citep{mondrian, minimaxmondrian}, whereas centered trees fail to reach it \citep{biau2012analysis,klusowski2021sharp}. Pointwise deviation bounds, and the geometric mechanism behind these discrepancies, have remained comparatively unexplored.

\noindent\textit{State of the art for Voronoi partitions.} The same question arises for prototype-based methods, with a different cell geometry. Two well-known such methods are \textit{Proto-NN} \citep{gyorfi2021universal} and \textit{OptiNet} \citep{NIPS2014_8c19f571, NIPS2017_934815ad, hanneke2021universal, kerem2023error}. Both construct a Voronoi partition from an auxiliary sample of prototypes, drawn independently of the responses, so as to favor cells that are more isotropic than the axis-aligned ones produced by tree constructions; OptiNet further imposes, through an $\eta$-net construction, a minimal spacing $\eta$ between prototypes, allowing for markedly finer control over cell volumes. Interestingly, Proto-NN and OptiNet both enjoy universality properties in general metric spaces, while the variant proposed in \cite{xue2018achieving}, called Proto-$k$-NN, fails to be universally consistent, for reasons similar to those explaining the failure of standard $k$-NN \citep{cerou2006nearest}. In finite dimension, convergence rates matching the minimax rate for Lipschitz functions have been obtained for OptiNet \citep{kerem2023error} and Proto-$k$-NN \citep{xue2018achieving,gyorfi2021universal}; for Proto-NN, obtaining rates has remained an open problem.

Viewed through the lens of shape regularity, these two families are no longer separate case studies but two answers to the same question, and the contributions of this work are organized accordingly.

\noindent \textbf{(i)} We first establish a new pointwise concentration inequality for the estimation error of partition-based regression, valid for a general (possibly random) partition. The bound displays two terms: a variance term scaling as the inverse square root of the number of points in the cell, and a bias term scaling as the diameter of the cell. It makes explicit that the geometry of the cell -- its volume relative to its diameter, as captured by the shape regularity condition of \cite{BPS} -- is the sole quantity governing the rate, and it serves as the common device through which all the constructions studied in this paper are subsequently analyzed.

\noindent \textbf{(ii)} Applying this device to trees, we obtain nonasymptotic deviation inequalities on the pointwise estimation error: the rate $n^{-1/(d+2)} \exp({2\sqrt{\log n \log\log n}})$ for centered trees, and a degraded rate involving $n^{-1/(\Theta d + 2)}$ for uniform trees where $\Theta \approx 5.5$ (Theorems \ref{cor_UcR} and \ref{cor_URT}). We show that these sub-optimal exponents and super-logarithmic corrections are unavoidable due to the inherent complexity of such constructions: with positive probability, independent of $n$, the aspect ratio $h_+(V)/h_-(V)$ grows exponentially in $\sqrt{N/d}$ (Propositions \ref{centered tree not regular} and \ref{uniform tree not regular}), so that neither construction is shape-regular. Mondrian trees \citep{mondrianroy2008,minimaxmondrian} escape this pitfall: selecting the splitting direction with probability proportional to the current side lengths forces the cell to remain nearly isotropic, guaranteeing shape regularity in probability (Proposition \ref{mondrian}) and, in turn, a genuinely nonasymptotic minimax rate (Theorem \ref{thm:mondrian}) -- though only with a probability whose decay is polynomial rather than exponential, a limitation that we show cannot be strengthened into an almost sure guarantee. This dichotomy illustrates a general principle: a splitting rule that is blind to the current shape of the cell produces an aspect ratio growing exponentially with non-negligible probability, so that achieving optimal rates requires breaking the independence between the splitting direction and the geometry of the cell -- whether through an appropriate stochastic mechanism such as Mondrian's, or through a data-driven adaptive rule in the spirit of CART.

\noindent \textbf{(iii)} Applying the same device to Voronoi partitions, we establish a nonasymptotic deviation bound for Proto-NN (Theorem \ref{th:main_th_proto} and Corollary \ref{cor:main_th_proto}), showing that the minimax optimal rate is achieved provided the number of prototypes $m$ is chosen appropriately as a function of $n$. To the best of our knowledge, no such result was previously available, and the problem is nontrivial, as pointed out by \cite{gyorfi2021universal}: ``\textit{Obtaining convergence rates for the universally consistent Proto-NN classifier (...) is currently an open research problem}''. We further obtain an analogous, slightly more favorable bound for OptiNet (Theorem \ref{th:main_th_OptiNet}): for a suitable choice of $\eta$ and $m$, it even yields an almost sure rate of order $(\log n/n)^{1/(d+2)}$. Notably, the success probability of the Proto-NN bound degrades exactly as it does for Mondrian trees, and we attribute both phenomena to the same cause: an auxiliary layer of randomness in the construction of the partition, independent of the responses, whose fluctuations occasionally produce cells of abnormally small size. The $\eta$-net step of OptiNet is precisely what suppresses these fluctuations -- confirming, from the Voronoi side, the principle identified for trees in (ii).

The paper is organized as follows. Section \ref{s2} sets up the regression framework and the partition-based regression estimator, states the general deviation bound, and recalls the shape regularity condition of \cite{BPS}, which is shown there to be necessary and sufficient, up to logarithmic factors, for achieving optimal rates. Sections \ref{sec_PRT} and \ref{s51} then apply this common framework to the two families of purely random partitions: Section \ref{sec_PRT} deals with purely random trees -- centered, uniform, and Mondrian -- while Section \ref{s51} develops the corresponding theory for Proto-NN and OptiNet.

\section{A pointwise deviation bound}\label{s2}

Let $d\geq 1 $ and $n \geq 1$. Let $ \left\{(X_i,Y_i) \in \mathbb R^d\times\mathbb R \,  : \, i=1,\ldots,n\right\}$ be a collection of random variables. Consider the following assumption.

\begin{enumerate}[label=(D), wide=0.5em,  leftmargin=*]
\item \label{cond:D} 
The random variables $\{(X,Y), (X_i,Y_i)_{i=1,\ldots,n} \}$ are independent and identically distributed with common distribution $  P$ on $\mathbb R^d\times\mathbb R$. 
\end{enumerate}
The vector $X\in \mathbb R^d$  is called the covariate vector and the variable $Y\in\mathbb R$ is called the response. The marginal distribution of $X$ is denoted  $P^X$ and its support is $S_X\subset\mathbb R^d$. Our goal is to estimate the regression function $(x\in S_X) \mapsto g(x)=\mathbb E[Y|X=x]$.

We consider partitioning estimators defined as follows. Let $\mathcal{V} = (V_j)_{j \in J}$ be a collection of measurable sets forming a partition of the covariate space $S_X $. For any $x \in S_X$, we denote by $\mathcal{V}(x)$ the unique cell of the partition $\mathcal{V}$ containing $x$. The regression estimator associated with the partition $\mathcal{V}$, denoted by $\hat g_{\mathcal V}$, is defined as:
$$\forall x \in S_X, \quad \hat g_{\mathcal V}(x) = \frac{\sum_{i=1}^n Y_i\mathds 1 _{ \mathcal V(x) }(X_i) }{\sum_{i=1}^n \mathds 1 _{ \mathcal V(x) }(X_i)  }\; ,$$
with the convention that $0/0 = 0$, which is in force throughout the rest of this work. The partition $\mathcal V$ might be random but it must be independent from the sample $(X_i,Y_i)_ {i=1,\ldots,n}$ or, more generally, measurable with respect to the $\sigma$-algebra generated by $X_1,\ldots, X_n$.

For any set $V\subset \mathbb R^d$, the diameter and the empirical volume of $V$ are defined respectively as
\begin{align*}
& \diam (V)  = \sup_{(x,y)\in V\times V} \|x-y\|_2, \qquad \qquad  P_n^X(V) =n^{-1} \sum_{i=1} ^ n \mathds 1 _{V}(X_i)  ,
\end{align*}
where $\|x\|_2^2 = \sum_{k=1} ^d x_k^2$. Throughout, $\lambda$ denotes the Lebesgue measure on $\mathbb{R}^d$ and $V_d$ the volume of the Euclidean unit ball. Let us recall the following classical result on the convergence of the estimation error associated to a partitioning estimator $\hat g_{\mathcal V}$. The next is an asymptotic result, one needs to consider a partition $\mathcal{V}_n $ that may depend on $n$.

\begin{theorem*}[Theorem 6.1 in \cite{devroye96probabilistic}]
    Let $n \geq 1$ and $d\geq 1$. Suppose that \ref{cond:D} is fulfilled and suppose that $Y\in \{0,1\}$. For each $n$, let $\mathcal{V}_n$ be a partition of $S_X$ which is $(X_1,\ldots, X_n)$-measurable and let $\mathcal V_n(x)$ denote the cell containing $x$.
    If $\diam(\mathcal V_n(X)) \to 0$, as $n\to\infty$, in probability, and $ n P_n^X(\mathcal{V}_n(X))    \to  \infty$, as $n\to\infty$, in probability, then 
    $$ \lim_{n \to + \infty} \mathbb E(  |\hat g_{\mathcal V}(X) - g(X)|  ) = 0.$$
\end{theorem*}

In the above, $L_1$ consistency is obtained by an easy adaptation of the proof of Theorem 6.1 in \cite{devroye96probabilistic} (considering bounded $Y$ instead of $Y\in \{0,1\}$). It is required for the number of points in the cell to diverge ($n P_n^X(\mathcal{V}_n(x)) \to \infty$) and the cell diameter to vanish ($\text{diam}(\mathcal{V}_n(x)) \to 0$). These two conditions illustrate a trade-off concerning the size of the cell that needs to be small enough but not too much. To effectively quantify this trade-off, we investigate a new result that extends the previous Theorem to the finite-sample setting that explicitly controls and quantifies these two competing effects. Two additional assumptions are needed. The first is the following sub-Gaussian assumption:

\begin{enumerate}[label=(E), wide=0.5em,  leftmargin=*]
\item \label{cond:epsilon} 
The random variable $\varepsilon  = Y - g(X)$ is sub-Gaussian conditionally on $X$ with parameter $\sigma^2$. That is, $\mathbb E [\varepsilon |X ]= 0$ and for all $\lambda \in \mathbb R$, $$ \mathbb E [ \exp( \lambda\varepsilon  ) |X ] \leq \exp\left( \frac{ \lambda ^2 \sigma^2}{ 2 }\right).$$
\end{enumerate}

This condition ensures that the regression noise $\varepsilon$ has light tails uniformly with respect to $X$. Note that this requirement accommodates general heteroscedastic structures: it allows $\varepsilon$ to depend on $X$ while ensuring integrability of all orders as well as an almost sure uniform bound on the conditional variance by $\sigma^2$. This represents a relaxation compared to standard frameworks that classically assume mutual independence between the error term and the predictors.

The second assumption is a Lipschitz regularity condition. For any subset $V\subset S_X$, we introduce the local Lipschitz constant $L(V)$ of $h:S_X\to \mathbb R$ as the infimal constant $L>0$ such that
$$| h(x) - h(y)| \leq L \Vert{}x-y\Vert{}_2 \quad \text{for all } (x,y) \in V^2.$$ 
A function $h:S_X\to \mathbb R$ is globally $L$-Lipschitz if $L=L(S_X)$. In this case, the local constant $L(V)$ satisfies $L(V)\leq L$ across any subset $V$. In what follows, for simplicity, we consider regression functions that are Lipschitz over the domain $S_X$ but the bounds are given with respect to the local Lipschitz constant $L(\mathcal V(x))$.

\begin{enumerate}[label=(L), wide=0.5em,  leftmargin=*]
  \item \label{cond:reg4} The function \( g: x \mapsto \mathbb E [ Y|X= x]\) is $L$-Lipschitz on \( S_X \). 
\end{enumerate}

Under these assumptions, we can now quantify the convergence rate of Theorem 6.1 in \cite{devroye96probabilistic} with the help of a non-asymptotic upper bound.

\begin{theorem}\label{sousgauss3}
Let $n\geq 1$, $d\geq 1$, $\delta\in (0,1/2)$ and $x\in S_X$. Assume that  \ref{cond:D}, \ref{cond:epsilon} and \ref{cond:reg4} are fulfilled. Let $\mathcal{V}$ be a partition of $S_X$ which is $(X_1,\ldots, X_n)$-measurable and let $\mathcal V(x)$ denote the cell containing $x$. We have with probability at least $1-2\delta$,
    $$ | \hat g_{\mathcal V}(x) - g(x)| \leq \sqrt{\dfrac{2\sigma^2 \log(1/\delta)}{n P_n^X(\mathcal{V}(x))} } + L(\mathcal{V}(x)) \diam(\mathcal{V}(x)).$$      
\end{theorem} 
The first term on the right-hand side represents the variance, which scales inversely with the local sample size $n P_n^X(\mathcal{V}(x))$, while the second term captures the approximation bias, scaling linearly with the cell's diameter. We thus recover the asymptotic conditions of \cite{devroye96probabilistic} as a direct consequence of this explicit risk bound. Note that when the partition $\mathcal V$ is based on another independent source of randomness, the results remain valid. 

In what follows, we will need an alternative statement valid for a fixed partition $\mathcal{V}$ (which is no longer $(X_1, \dots, X_n)$-measurable) -- or also based on another independent source of randomness -- with $nP^X$ instead of the number of points $n P_n^X$. This is the object of the following corollary.
\begin{corollary}\label{sousgauss}
Let $n\geq 1$, $d\geq 1$, $\delta\in (0,1/3)$ and $x\in S_X$. Assume that  \ref{cond:D}, \ref{cond:epsilon} and \ref{cond:reg4} are fulfilled. Let $\mathcal{V}$ be a partition of $S_X$, with $\mathcal{V}(x) \in \mathcal{V}$ being the measurable set containing $x$. If $ nP^X(\mathcal{V}(x))  \geq 8 \log(1/\delta) $, we have with probability at least $1-3\delta$,
    $$ | \hat g_{\mathcal V}(x) - g(x)| \leq \sqrt{\dfrac{4\sigma^2 \log(1/\delta)}{n P^X(\mathcal{V}(x))} } + L(\mathcal{V}(x)) \diam(\mathcal{V}(x)).$$      
\end{corollary}

Similar but different bounds are established in \cite{BPS}. They are valid for the supremum norm over $x\in S_X$ at the expense of a logarithmic factor and a Vapnik-type regularity assumption on the cells. 

To optimize the obtained bound in the previous corollary, we must balance the two terms $\diam(\mathcal{V}(x))$ and $n P^X(\mathcal{V}(x))$. A central further question in our analysis is whether these $\mathcal{V}(x)$ satisfy the shape regularity condition introduced in \cite{BPS}. This property ensures that the cells do not collapse disproportionately in certain dimensions. We restate the formal definition below for a general set $V\subset \mathbb R^d$ first, and then for a hyper-rectangle $A \subset S_X$, for which $h_-(A)$ and $h_+(A)$ denote the smallest and largest side lengths, respectively.

\begin{definition}\label{def:gamma_regular}\label{def:beta_regular}
For $\gamma>0 $, a set $ V $ is called $\gamma$-shape-regular ($\gamma$-SR) if
$\diam(V)^d \leq \gamma \lambda (V)$. For $\beta>0 $, a hyper-rectangle $ A$ is called $\beta$-shape-regular ($\beta$-SR) if 
$ h_+ (A)  \leq \beta h_-(A)$.
\end{definition}

When $ \mathcal V(x)$ is $\gamma$-SR, using the approximation that $P^X (\mathcal  V(x))\simeq \lambda (\mathcal  V(x)) $, the upper bound of Corollary \ref{sousgauss} becomes 
$$ \sqrt{\dfrac{4\sigma^2 \log(1/\delta)}{n \lambda(\mathcal{V}(x))} } + L(\mathcal{V}(x)) \, (\gamma \lambda (\mathcal{V}(x)))^{1/d} . $$
Therefore, the optimal rate of convergence can be obtained by allowing for a fine-tuning of $\lambda (\mathcal{V}(x))$. In contrast, if the factor $\gamma $ grows to infinity with the sample size, then the obtained upper bound fails to reach the optimal rate. For this reason, throughout the paper, we study several examples by comparing the SR factor obtained after a careful analysis of $\diam(\mathcal V(x))$ and  $P^X(\mathcal V(x))$.

\section{Purely random trees}\label{sec_PRT}

\subsection{Background}

In this Section, we introduce a minimal mass assumption, specifically tailored for the analysis of tree-based regression estimators. It imposes a uniform lower bound on the density over the unit hypercube.

\begin{enumerate}[label=(XTREE), wide=0.5em, leftmargin=*]\item \label{cond:density_XCART} The random variable $X$ admits a density function $f_X$ on $S_X = [0,1]^d$ which is bounded from below by $b >0$, i.e., $f_X(x) \geq b$ for all $x\in S_X$.\end{enumerate}
In this Section, we analyze partition-based estimators where $\mathcal{V}$ is constructed via a purely random tree process. In this framework, the splitting directions and positions are chosen according to a fixed distribution, independent of the response variables in the training set. More precisely, we consider purely random trees (PRT), that are built by successively refining a partition of the space, in a way that is independent of the initial sample $\{ (X,Y), (X_1,Y_1),\ldots, (X_n, Y_n) \} $. We will now analyze specific tree structures, starting with centered trees, followed by uniform trees, and finally Mondrian trees. Moreover, for clarity, we assume that $S_X = [0,1]^d$ and we always take $x \in S_X$.

To set up notations, let us describe a PRT locally around a point $x$. The tree is generated iteratively, and at each step $i$, for the cell $\mathcal V(x)$ containing $x$, a coordinate is selected according to a random variable $D_i\in \{ 1,\ldots,d\}$ and then the side of the cell in direction $D_i$, that we write $(a,b)$, $a<b$, is split into two intervals $(a,a+(b-a)S_i)$ and $(a+(b-a)S_i,b)$, thus defining two new cells $C_1$ and $C_2$. Consequently, each step $i$ consists in splitting a cell and depends on a pair of random variables $(D_i,S_i)$, that is independent of the dataset $\{(X,Y), (X_i,Y_i)_ {i=1,\ldots,n} \}$. After $N$ steps, we denote $\mathcal V(x)=\mathcal V (x, (D_i,S_i)_{i=1}^N)$. Moreover, we denote by $\bar{S}_i$ the length reduction of the side $D_i$ of the considered cell at step $i$, that is either equal to $S_i$ or $1-S_i$ according to the fact that the coordinate $x_{D_i}$ is smaller or greater than $a+(b-a)S_i$, respectively. Then, we have for all $i \in \{1,\dots, N\}$, $\min(S_i, 1-S_i) \leq \bar{S}_i \leq \max(S_i, 1-S_i)$. Note that for centered trees, $S_i=1/2$ almost surely, then $\bar{S}_i = 1/2$, and consequently the Lebesgue volume of the cell containing $x$ after $N$ steps is equal to $1/2^N$.

\subsection{Centered random trees}\label{subsec_centered}

Let us first provide some deviation bounds for the diameter of the cell $\mathcal{V}(x)$ built with centered random trees. Moreover, in the case of centered random trees, the volume of the cell $\mathcal{V}(x)$ after $N$ steps is simply $  \lambda(\mathcal{V}(x))= (1/2)^N$. The diameter of $\mathcal{V}(x)$ behaves as follows.

\begin{proposition}\label{prop:diam_CRT}
Let $d\geq 2$ be an integer. Consider that $S_i=1/2 $ almost surely and that $D_i$ are independent of each other and uniformly distributed over $\left\{ 1,\ldots,d\right\}$. Then, for $\mathcal V (x)= \mathcal V(x, (D_i,S_i)_{i=1}^N)$ and for any $\alpha \in (0,1/d)$,
\begin{equation*}
    \PP(\diam(\mathcal{V}(x))\geq \sqrt{d}2^{-\alpha N}) \leq d\left(1-\frac{1-\theta}{d}\right)^N \theta^{-\alpha N},
\end{equation*}
where $\theta=(d-1)\alpha/(1-\alpha).$ Moreover, for any $\alpha \in (1/d,1)$, we have for the same $\theta$
\begin{equation*}
    \PP(\diam(\mathcal{V}(x))\leq \sqrt{d}2^{-\alpha N}) \leq d\left(1-\frac{1-\theta}{d}\right)^N \theta^{-\alpha N}.
\end{equation*}
\end{proposition}

We have the following Proposition about the shape-regularity of centered random trees.

\begin{proposition}\label{cor22}
When the number of splits goes to infinity, it holds that, almost surely, there exists $N_0\geq 1$ such that for all $N\geq N_0$,
$$ \sqrt{d} 2^{-N/d-2\sqrt{(d-1)N\log(N)/d^2}} \leq \diam(\mathcal{V}(x)) \leq \sqrt d 2^{-N/d+2\sqrt{(d-1)N\log(N)/d^2}}.$$
In addition, if we denote the normalized diameter by $\diam^{\#}(\mathcal{V}(x)) := \diam(\mathcal{V}(x))/\sqrt{d}$, almost surely it holds, for $N$ large enough,
\begin{equation*}
  2^{\, -2\sqrt{(d-1)N\log(N)}}   \leq \dfrac{\diam^{\#}(\mathcal{V}(x))^d}{\lambda(\mathcal{V}(x))} \leq 2^{ \, 2\sqrt{(d-1)N\log(N)}}. 
\end{equation*}
\end{proposition}

The previous results are valid in any dimension $d\geq 1$, but in dimension one, the (normalized) diameter of any cell is always equal to its Lebesgue volume, so we always have $\diam(\mathcal{V}(x))=\diam^{\#}(\mathcal{V}(x))=\lambda(\mathcal{V}(x))$.

We deduce  the following high probability upper bound on the pointwise error of the resulting regression estimator. 

\begin{theorem}\label{cor_UcR}
Let  $n\geq 1$, $d\geq 1$ and $x\in S_X$. Assume that the integer $N$ is such that $N=d\log(n)/((d+2)\log(2))$. Suppose that $\mathcal V (x)= \mathcal V(x, (D_i,S_i)_{i=1}^N)$ is obtained from a centered random tree as described in Proposition \ref{prop:diam_CRT}. Under \ref{cond:D}, \ref{cond:epsilon}, \ref{cond:reg4} and \ref{cond:density_XCART}, there exists $C>0$, that only depends on the parameters of the problem but not on $n$, such that with probability $1$, there is $n_0$ such that for all $ n \geq n_0$,
\[
 | \hat g_{\mathcal V}(x) - g(x)| \leq
 C \,  n^{-1/(d+2)}e^{2\sqrt{\log(n)\log\log (n)}}.
\]
\end{theorem}

From Theorem \ref{cor_UcR}, we see that the estimator based on the centered random partition achieves a pointwise estimation error that is close to the minimax rate $n^{-1/(d+2)}$ for the error in expectation, in the sense that for any $\varepsilon>0$, almost surely, for $n$ large enough, the estimation error is smaller than $n^{-1/(d+2)+\varepsilon}$.

To conclude our analysis of centered trees, we also include the following negative result, which establishes that centered trees are not shape-regular, as suggested by the sub-optimality of the convergence rate obtained in Theorem \ref{cor_UcR}. 

\begin{proposition}\label{centered tree not regular}
Let $d \geq 2$. Centered trees are not $\beta$-SR, i.e.,
for any \( N \geq d \) and any hyper-rectangle $\mathcal V (x)= \mathcal V(x, (D_i,S_i)_{i=1}^N)$ obtained from a centered random tree, as described in Proposition \ref{prop:diam_CRT}, we have, with probability at least $1/14$,
$$\dfrac{h_+(\mathcal V (x))}{h_-(\mathcal V (x))} \geq 2^{\sqrt{N/d}}.$$
\end{proposition}

While, in the above, the value $1/14$ can certainly be improved, we stress that our result implies that shape-regularity fails to happen on an event having positive probability (independent of $n$).

\subsection{Uniform random trees}\label{subsec_URT}

In the context of purely random trees, the exact position of the target point $x \in S_X$ within the current cell is unknown and complex to track dynamically. Consequently, at any given split step $i$, we cannot determine almost surely whether $x$ falls into the left or the right child cell. This means the exact value of the relative reduction $\bar{S}_i$ (which takes either the value $S_i$ or $1-S_i$) remains intractable. However, the almost sure bounds $\min(S_i, 1-S_i) \leq \bar{S}_i \leq \max(S_i, 1-S_i)$ provide a powerful geometric safeguard: they allow us to uniformly bound the volume and diameter reductions of the cell $\mathcal{V}(x)$ across all steps $i$, for any $x$, without requiring any knowledge of its spatial path through the tree.

\begin{proposition}\label{prop:diamURT2}
Consider  that $S_i$ are independent and uniformly distributed over $(0,1)$ and that $D_i$ are independent of each other and from the $S_i$'s and uniformly distributed over $\left\{ 1,\ldots,d\right\}$. Then, for $\mathcal V (x)= \mathcal V(x, (D_i,S_i)_{i=1}^N)$ and for any $\theta \geq 0$,
\begin{equation*}\label{eq_dev_diam_URT}
    \PP(\diam (\mathcal{V}(x))\geq \sqrt{d}e^{-N(1-\log(2))/d+N\theta}) \leq de^{-Nd\theta^2/2}.
\end{equation*} 
Moreover, for all $\theta \in (0,6/d)$ we have,
\begin{equation*}
    \PP(\diam (\mathcal{V}(x))\leq \sqrt{d}e^{-N(1+\log(2))/d - N\theta}) \leq de^{-Nd\theta^2/24}.
\end{equation*} 
\end{proposition}

\begin{proposition}\label{prop:diam_vol_URT2}
Consider  that $S_i$ are independent and uniformly distributed over $(0,1)$ and that $D_i$ are independent of each other and from the $S_i$'s and uniformly distributed over $\left\{ 1,\ldots,d\right\}$. Then, for $\mathcal V (x)= \mathcal V(x, (D_i,S_i)_{i=1}^N)$ and for any $\alpha >1$,
\begin{equation*}\label{eq_dev_vol_URT}
\PP\left(\lambda(\mathcal V(x))\leq e^{-N(\alpha + \log(2))}\right)\leq (\alpha e^{1-\alpha})^N.
\end{equation*}
In addition, for any $\alpha \in (0,1)$,
\begin{equation*}
 \mathbb{P}\left(\lambda(\mathcal{V}(x)) \geq e^{-N(\alpha - \log(2))}\right) \leq (\alpha e^{1 - \alpha})^N.
\end{equation*}
\end{proposition}

\begin{proposition}\label{cor_diamvolURT}
When the number of splits goes to infinity, it holds that, almost surely, there exists $N_0\geq 1$ such that for all $N\geq N_0$,
$$ \sqrt d e^{-N(1+\log(2))/d-4\sqrt{3N\log(N)/d}} \leq \diam(\mathcal{V}(x)) \leq \sqrt d e^{-N(1-\log(2))/d+2\sqrt{N\log(N)/d}}$$
and
$$ e^{-N(1+\log(2)) - 2\sqrt{N\log(N)}} \leq \lambda(\mathcal{V}(x)) \leq e^{-N (1 - \log(2)) + 2\sqrt{N\log(N)}}.$$
As a consequence, if we denote the normalized diameter $\diam^{\#}(\mathcal{V}(x)) := \diam(\mathcal{V}(x))/\sqrt{d}$, we obtain that, almost surely, for $N$ large enough, 
\begin{equation*}
  e^{-2N\log(2) - 2\sqrt{N\log(N)}(1+2\sqrt{3d})} \leq \dfrac{\diam^{\#}(\mathcal{V}(x))^d}{\lambda(\mathcal{V}(x))} \leq e^{2N\log(2) + 2\sqrt{N\log(N)}(1+\sqrt{d})}.
\end{equation*}
\end{proposition}

In the same spirit as for centered random trees, we obtain an upper bound on the pointwise error of the resulting regression estimator. 

\begin{theorem}\label{cor_URT}
   Let $n\geq 1$, $d\geq 1$, $x\in S_X$ and $C_d := (1+\log(2))d+2(1-\log(2))$. Assume that the integer $N$ is such that $N= d\log(n)/C_d$. Suppose that $\mathcal V (x)= \mathcal V(x, (D_i,S_i)_{i=1}^N)$ is obtained from a uniform random tree as described in Proposition \ref{prop:diamURT2}. Under \ref{cond:D}, \ref{cond:epsilon}, \ref{cond:reg4} and \ref{cond:density_XCART}, there exists $\bar{C}>0$, that only depends on the parameters of the problem but not on $n$, such that almost surely, there exists an integer $n_0$ such that for all $ n \geq n_0$,
\[
 | \hat g_{\mathcal V}(x) - g(x)| \leq
 \bar{C} \, n^{-1/(\Theta d+2)} e^{2\sqrt{\log(n)\log\log (n)}},
\]
where $\Theta := (1+\log(2))/(1-\log(2))$.
\end{theorem}

Ignoring the sub-polynomial factor, the rate delivered by our bound is governed by $n^{-1/(\Theta d + 2)}$ where $\Theta \approx 5.5$. If we had $\Theta = 1$, we would essentially recover the standard optimal minimax rate $n^{-1/(d+2)}$ (again up to the sub-polynomial factor). This is what happens for centered trees, where $\bar S_i = 1/2$ is deterministic, so that $\lambda(\mathcal V(x))$ and $\mathrm{diam}(\mathcal V(x))^d$ decay at the same exponential rate, their ratio is nonetheless unbounded, as Proposition \ref{cor22} shows, which is what the sub-polynomial correction in Theorem \ref{cor_UcR} accounts for. For uniform trees, in contrast, the constants appearing in the ratio $\Theta = (1+\log(2))/(1-\log(2))$ stem directly from the lower and upper bounds of the cell sizes. Specifically, $1+\log(2)$ corresponds to the geometric decay rate of the smallest fraction $m_i = \min(S_i, 1-S_i)$, while $1-\log(2)$ corresponds to that of the largest fraction $M_i = \max(S_i, 1-S_i)$.

As we did previously for centered trees, we would like to study the shape regularity of the specific cell $\mathcal{V}(x)$ containing the target point $x \in S_X$. However, we face a major difficulty due to $x$: $\mathcal{V}(x)$ is a random cell whose construction path dynamically depends on the spatial position of $x$. Consequently, the same technique as before using the Paley-Zygmund inequality does not apply directly to $\mathcal{V}(x)$ because the successive relative splits are correlated. On the other hand, one can perfectly apply the Paley-Zygmund inequality to a fixed cell $V$ of the partition (i.e., a cell constructed by following a predetermined path, independently of any evaluation point $x$). Doing so reveals that uniform random splits naturally and frequently generate highly anisotropic (elongated) shapes. Consequently, uniform trees fundamentally fail to satisfy the shape regularity property, as formalized in the following proposition.

\begin{proposition}\label{uniform tree not regular}
Let $d \geq 2$. Uniform trees are not $\beta$-SR, i.e., for any $N \geq d$ and any fixed hyper-rectangle $V \in \mathcal V$ constructed by a predetermined sequence of $N$ splits independent of the evaluation point, we have, with probability at least $1/11$,
$$\dfrac{h_+(V )}{h_-(V )} \geq e^{\sqrt{N/d}}.$$
\end{proposition}

As for Proposition \ref{centered tree not regular} above, the precise value $1/11$ does not play any crucial role here. The important fact is that Proposition \ref{uniform tree not regular} shows that shape-regularity is violated on an event having positive probability (independent of $n$).

\subsection{Mondrian trees}\label{s43}

 A Mondrian process \textit{MP} is a process that generates infinite tree partitions of $S_X = [0,1]^d $ (\cite{mondrianroy2008}). These partitions are built by iteratively splitting the different cells at random times, where both the timing and the position of the splits are determined randomly. Additionally, the probability that a cell is split depends on the length of its sides, and the probability of splitting a particular side is proportional to the length of that side. Once a side is selected, the exact position of the split is chosen uniformly along that side.
We can then define the pruned Mondrian process \textit{MP}($\Lambda$). This version introduces a pruning mechanism that removes splits occurring after a specific time $\Lambda > 0$, which is referred to as the lifetime.

Mondrian trees are studied in detail in the paper \cite{minimaxmondrian}. In particular, it is possible to give a simple description of the distribution of the cell $\mathcal V(x)$ containing $x$ and generated by a process \textit{MP}($\Lambda $). Such a property helps to demonstrate the following result, that Mondrian trees are \(\beta\)-SR in probability.

\begin{proposition}\label{mondrian}
 For any $x\in [0,1] ^d$, let  $\mathcal V (x)$ be the hyper-rectangle containing $x$ obtained from a \textit{MP}($\Lambda $) tree. For $\delta \leq \min(1 - e^{-\Lambda d}, \; 1 - (1 - e^{-1})^d)$, we have, with probability at least \(1 - 2\delta\),
\[ \frac{h_+(\mathcal V (x))}{h_-(\mathcal V (x))} \leq \frac{5d \log(\delta/d)}{{\log(1-\delta)}}. \]
\end{proposition}

The latter inequality implies that, for any small $\delta>0$, there is a constant $K_\delta >0 $ such that  the event $h_+(\mathcal V (x)) \leq K_\delta \, h_-(\mathcal V (x))$ occurs with probability at least $1-\delta$. In other words,  ${h_+(\mathcal V (x))} /{h_-(\mathcal V (x))}~ $ is a tight sequence. As a consequence, Mondrian regression trees attain, with high probability, the minimax rate for the pointwise error in expectation.

\begin{theorem}\label{thm:mondrian}
Let  $n\geq 1$, $d\geq 1$, $x\in S_X$, $\Lambda  \asymp n^{1/(d+2)}$, $\delta\in (0,1/5)$ and define $c_{\delta,d} = \log(1/\delta) (d/ \log(1/(1-\delta)) )^d$. Let  $\mathcal V (x)$ be the hyper-rectangle  containing $x$ obtained from a \textit{MP}($\Lambda $) tree.
Under \ref{cond:D}, \ref{cond:epsilon}, \ref{cond:reg4} and \ref{cond:density_XCART}, if $n^{2/(d+2)} b \geq 8 c_{\delta,d} $, then it holds, with probability at least $1-5\delta$, 
\[
 | \hat g_{\mathcal V}(x) - g(x)| \lesssim
 C \, n^{-1/(d+2)}\; ,
\]
where $C = \sqrt{ {4\sigma^2  c_{\delta,d}} / {b }} +  5 \sqrt{d} \, L(\mathcal{V}(x))\log(d/\delta)$. 
\end{theorem}

While the convergence rate above matches the minimax rate for pointwise error in expectation, it holds with a probability that scales poorly, far from exponential decay. For instance, this rate cannot be extended to an almost sure convergence guarantee. Interestingly, an almost identical result, minimax rates under poor probability scaling, has been obtained for Proto-NN in Corollary \ref{cor:main_th_proto}. This similarity likely stems from the use of an additional source of randomness in the partition construction of both Proto-NN and Mondrian tree. We believe that in both cases, the (random) construction process may lack sufficient stability, with bad events occurring with too large probability, such as the formation of excessively small cells.

The analysis of Section D in \cite{BPS} demonstrates that structural sub-optimality is inevitable when splitting directions are chosen independently of the cell's geometry. As highlighted in \cite{BPS}, trees lacking such internal dependencies, where the choice of the dimension to split does not account for the current side lengths, suffer from a shape regularity ratio that grows at an exponential rate with a probability bounded away from zero. In contrast, Mondrian trees manage to escape this fate. Although they are constructed independently of the labels $Y$, they are not ``blind'' to the shape of the cell. In a Mondrian process, the direction $j$ is selected with a probability proportional to the current side length $h_j$. This internal dependence acts as a regulatory mechanism: the longer a side is, the more likely it is to be split, which mechanically forces the cell to remain nearly isotropic. This self-correcting feature ensures that the shape regularity ratio remains controlled, avoiding the exponential explosion described in Proposition 22 in \cite{BPS}. Ultimately, this underscores a major conclusion: to achieve optimal convergence rates, an algorithm must break the independence between splitting direction and cell dimensions, either through an appropriate stochastic mechanism like Mondrian trees or via a data-driven adaptive procedure such as an amended version of CART (See section 5.3 in \cite{BPS}).

\section{Prototype nearest neighbors and OptiNet}\label{s51}

Building upon the study of purely random trees, we now introduce Prototype Nearest Neighbors (Proto-NN) and OptiNet. These methods share a fundamental similarity with purely random trees: they all belong to the class of local averaging estimators where the prediction at a given point is determined by $\mathcal{V}$. In their most basic form, both purely random trees and Proto-NN models can be constructed through a ``blind'' stochastic process, the former by choosing split positions independently of the labels, and the latter by randomly sampling prototype locations in the feature space. However, these models differ significantly in their geometric flexibility and their path toward adaptivity. While purely random trees are constrained by a rigid, axis-aligned recursive structure that often leads to a collapse in shape regularity, Proto-NN models employ Voronoi partitions, which naturally promote more isotropic cells. OptiNet refines this approach by introducing a rigorous selection of prototypes through the construction of an $\eta$-net. Instead of allowing randomness to dictate the proximity of the centers, OptiNet imposes a minimum distance $\eta$ between each prototype, thereby ensuring that the Voronoi cells do not clump together and keep a $d$-th power of the diameter proportionate to their volume.

Recall that the observed sample $ \{ (X,Y) , (X_1,Y_1),\ldots, (X_n, Y_n) \}  $, as introduced in Section \ref{s2}, is independent and identically distributed with covariates $X_i\in \mathbb R^d$ and response variable $Y_i\in \mathbb R$. The goal is to build a regression map to estimate $\mathbb E [Y|X=x]$, for a given $x\in \mathbb R^d$. A prototype learning algorithm relies on two steps: construct the prototype sample $(Z_j,\overline {Y}_j)_{j=1,\ldots, m}$ and train a learning rule based on the prototypes. Arguably the simplest approach among the $1$-nearest neighbor prototype learning is the one studied in \cite{gyorfi2021universal}, called Proto-NN, where the prototype covariates collection $(Z_j)_{j=1,\ldots, m}$ forms an independent and identically distributed collection of random variables with the same distribution as $X$. The labels $(\overline{Y}_j)_{j=1,\ldots, m}$ are created using the initial sample $(X_i, Y_i) _{i=1,\ldots, n}$ as follows: for each $j=1,\ldots, m$,
$$\overline{Y}_j = \frac{\sum_{i=1} ^n  Y_i \ind_{ V_j } (X_i) }{\sum_{i=1} ^n  \ind_{ V_j } (X_i) } $$ where $(V_j)_{j=1,\ldots, m}$ denotes the Voronoi cells of $(Z_j)_{j=1,\ldots, m}$. The resulting algorithm is the $1$-NN rule applied to the prototype sample $(Z_j,\overline{Y}_j)_{j=1,\ldots, m}$, as formally introduced below.

For $x\in \mathbb R^d$, let $\hat X_1(x)$ be the nearest neighbor to $x$ among the $ Z_1,\ldots, Z_m$, where tie breaking is done, for instance, by favoring larger indexes so that a unique $\hat X_1(x)$ is identified for each $x$. Let $V_k$ denote the Voronoi cell of $Z_k$ defined as $ \{x\in S_X \, : \, \hat X_1(x) = Z_k \}$. The collection $V_1,\ldots, V_m$ forms a partition of the domain $S_X$. Therefore, each $x$ can be given a unique element $ \mathcal V(x) := V_j$ whenever $x\in V_j $. 

The Proto-NN prediction rule then writes
$$\forall x \in S_X, \quad \hat g _{\textrm{proto}}(x) = \frac{\sum_{i=1} ^n Y_i  \ind_{ \mathcal  V (x) } (X_i )  }{\sum_{i=1} ^n    \ind_{\mathcal  V(x) } (X_i)  }\; ,$$
with, as usual, the convention that $0/0 = 0$. Note from its definition that Proto-NN is an estimator that results from a random partition.

Let us now introduce the assumptions required to establish our concentration bound for Proto-NN. First, we require the prototype variables $Z_i$ to satisfy the following condition.

\begin{enumerate}[label=(DZ), wide=0.5em,  leftmargin=*]
\item \label{cond:D1} 
The random variables $\{Z, (Z_i)_ {i=1,\ldots,m} \}$ are independent and identically distributed on $\mathbb R^d$ with common distribution $  P^Z  = P^X$. 
\end{enumerate}

Next, we describe the assumptions on the distribution of the covariates $X$ which also apply to the prototypes $Z$ because $P^X = P^Z$.

\begin{enumerate}[label=(XZ), wide=0.5em,  leftmargin=*]
\item \label{cond:reg2} There is $\rho_0 > 0$ such that $S_X\subset B (0, \rho_0)$. Moreover, there is $c_d>0$ and $T_0>0$ such that
\begin{align*}
&\lambda (S_X \cap  B(x, \tau ) ) \geq c_d \lambda   ( B(x, \tau )) , \qquad \forall \tau \in (0,T_0] , \, \forall x\in S_X.
\end{align*}
 By changing \( c_d \), we can assume that \( T_0 = \rho_0\). Additionally, $X$ has a density $f_X$ on $S_X$  and there exist constants $0 < b\leq M <+\infty$ such that $ b \leq f_X (x) \leq M$. 
\end{enumerate}

\begin{enumerate}[label=(XNN), wide=0.5em,  leftmargin=*]
\item  \label{cond:density_XNN} There is a positive function $\ell $ defined on $S_X$ and $T_0>0$ such that, for all $x\in S_X$ and $\tau \in (0, T_0)$,
$$ P^X( B(x, \tau) )  \geq   \ell(x)\tau ^d.$$
\end{enumerate}

In Assumption \ref{cond:D1}, we require that the $X_i$'s and $Z_i$'s follow the same distribution. This is indeed a classical framework for prototype algorithms. Interestingly, we note that, in fact, the distributions $P^Z$ and $ P^X$ need not be identical to preserve the rates exhibited in Theorem \ref{th:main_th_proto} and Corollary \ref{cor:main_th_proto}. More precisely, if the distributions $P^Z$ and $P^X$ are different, if $P^Z$ satisfies Assumption \ref{cond:reg2} and if $P^X$ follows Assumption \ref{cond:density_XNN} as for the classical $k$-NN estimator, then the results of Theorem \ref{th:main_th_proto} and Corollary \ref{cor:main_th_proto} would remain unchanged, up to constants. In other words, the rates are preserved under a distribution shift for $P^Z$ and $P^X$, if they respectively satisfy Assumptions \ref{cond:reg2} and \ref{cond:density_XNN}. 

Note also that, compared to Assumption \ref{cond:density_XNN} used in the previous analysis of $k$-NN regression in \cite{BPS}, Assumption \ref{cond:reg2} is slightly stronger. The latter assumption is needed in our proofs to ensure that the (Lebesgue) volume of the Voronoi cell $\mathcal V(x)$ is large enough. 

We also adapt the sub-Gaussian noise assumption to account for the presence of the prototype sample.

\begin{enumerate}[label=(EZ), wide=0.5em,  leftmargin=*]
  \item  \label{cond:reg3} The random variable $\varepsilon$ is sub-Gaussian conditionally on $X$ and $Z$ with parameter $\sigma^2$.
\end{enumerate}

We are now ready to state our non-asymptotic error bound for Proto-NN.

\begin{theorem}\label{th:main_th_proto}
     Assume that \ref{cond:D}, \ref{cond:D1}, \ref{cond:reg4}, \ref{cond:reg2} and \ref{cond:reg3} are fulfilled. Let $\delta \in (0,1/5)$ and $x \in S_X$. If $n \geq 1,m \geq 3$ are such that $$\dfrac{n}{m} \geq \dfrac{8\psi_d\log(1/\delta)}{ \delta^d}, \quad \text{ and } \quad 32 d \log(12m / \delta )  \leq T_0^d  m b c_d V_d, $$
where $\psi_d = (2^d18Md)^d (2c_d b)^{-d}$, then with probability at least $1- 5\delta$,
    $$|\hat g_{\textrm{proto}} (x) - g(x) | \leq\sqrt{\dfrac{4\sigma^2 \psi_d\log(1/\delta)}{\delta^d }} \sqrt{\dfrac{m}{n}} + 2 L(\mathcal{V}(x))  \left(\frac{ 32d \log(12m/\delta) }{ mb c_d V_d }  \right)^{1/d}.$$
\end{theorem}

By choosing $m$ appropriately as a function of $n$, the following minimax convergence rate is established.

\begin{corollary}\label{cor:main_th_proto}
Let $\delta \in (0,1/5)$ and $x \in S_X$. In Theorem \ref{th:main_th_proto}, if $n$ is sufficiently large, then choosing the integer $m$ such that $m \asymp (n/c_\delta) ^{d/(d+2)} \log(12n/\delta)^{2/(d+2)}$, with $c_\delta =  \log(1/\delta) / \delta^d $, yields the following inequality with probability at least $1-5\delta$,
    $$|\hat g_{\textrm{proto}} (x) - g(x) | \lesssim  C \left( \frac{c_\delta \log(12n/\delta)}{n }\right)^{1/(d+2)}$$
    where $$C = \sqrt{4\sigma^2 \psi_d} +  2 L(\mathcal{V}(x)) \left(\dfrac{ 32d }{ b c_d V_d }\right)^{1/d}  \ \ \text{   and   } \ \ \psi_d =  \left(\dfrac{2^{d-1}18Md}{c_d b}\right)^d.$$
\end{corollary}

The previous results are, to the best of our knowledge, the first concentration bounds on the error of the Proto-NN regression estimator.  As pointed out in \cite[Section 3]{gyorfi2021universal}, ``obtaining convergence rates for the universally consistent Proto-NN classifier [...] is currently an open research problem'', that the authors bypass by considering another algorithm that is simpler to analyze and that they term ``Proto-$k$-NN''. 

The key step in the proof is to get a lower bound on $P^X ( \mathcal V (x)) $. This step involves the control of some order statistics of the distances between pairs of prototype variables. The analysis exhibits a quite poor scaling -- i.e., far from exponential -- of the probability at which the minimax rate holds. A similar situation is observed for Mondrian trees and we believe that this cannot be much improved for these estimators. Modifying the definition of the Proto-NN estimator in order to improve the probability bound will be the subject of a forthcoming work. 

Concerning the shape regularity theory developed in previous sections, the Proto-NN algorithm is based on a \(\gamma\)-regular cell $\mathcal{V}(x)$ with high probability, in the sense that there exists a constant \(c  > 0\) such that, with probability at least \(1 - 2\delta\),  
$$\diam(\mathcal V(x))^d \leq c \frac{\log(m/\delta)}{\delta^d} \lambda(\mathcal V(x)).$$ 
This implies that \(\gamma\)-regularity (in probability) holds with a parameter \(\gamma = c \log(m/\delta)/\delta^d\) that is polynomial in $1/\delta$, which is in line with the poor scaling of the probability rate in the concentration bound of Theorem \ref{th:main_th_proto}.

Another approach studied in \cite{NIPS2017_934815ad,hanneke2021universal} and called OptiNet, consists in creating  prototype covariates $(Z_j(\eta))_{j=1,\ldots, m(\eta)} $ as a maximal $\eta$-net  subset of $(Z_j)_{j=1,\ldots, m}$, for which the minimum spacing between the elements, $\min_{j\neq k} \|Z_j(\eta)- Z_k(\eta)\|$, is larger than $\eta$. The prototype labels are then created in the same way as for Proto-NN, by averaging the labels inside the Voronoi cells obtained from $(Z_j(\eta))_{j=1,\ldots, m(\eta)} $. Let $\mathcal V_\eta (x) $ be the Voronoi cell of $x$, with respect to the sample $Z(\eta)=(Z_i(\eta))_{i=1,\dots,m(\eta)} $. The OptiNet prediction rule is given by $$\forall x \in S_X, \quad \hat g _{\textrm{opt}}(x) = \frac{\sum_{i=1} ^n Y_i  \ind_{ \mathcal  V_\eta(x) } (X_i )  }{\sum_{i=1} ^n    \ind_{\mathcal  V_\eta(x) } (X_i)  }$$ with the convention that $0/0 = 0$. For the OptiNet algorithm, we obtain the following error bound.

\begin{theorem}\label{th:main_th_OptiNet}
    Suppose that \ref{cond:D}, \ref{cond:D1}, \ref{cond:reg4}, \ref{cond:reg2} and \ref{cond:reg3} hold true. Let $\delta \in (0,1/4)$ and $x \in S_X$. If $n \geq 1$, $m \geq 1$ are such that 
    \begin{equation*}
     32d\log(12m/\delta) \leq  T_0^d  m b c_d V_d, \quad \eta \leq 2T_0 \quad \text{ and } \quad n b V_d c_d \eta^d \geq 2^{d+3} \log(1/\delta), 
    \end{equation*}
then with probability at least $1-4\delta$,
        $$|\hat g_{\textrm{opt}} (x) - g(x) | \leq\sqrt{\dfrac{2^{d+2} \sigma^2 \log(1/\delta)}{n b V_d c_d \eta^d} }+ 2 L(\mathcal{V}_\eta(x)) \left\{  \eta +  \left(\dfrac{32d \log(12m/\delta)}{mb c_d V_d} \right)^{1/d} \right\}.$$
\end{theorem}

Note that the OptiNet algorithm has the same rate of convergence as Proto-NN, but the above upper bound holds with a higher probability compared to the one of Proto-NN. This is a consequence of the $\eta$-net construction, which allows the control of the volume of the Voronoi cells, that is larger than $\eta^d$, in a better way than for Proto-NN. 
Optimizing in $\eta$ and $m$ the upper bound, the order of the optimal choice corresponds to $\eta = n^{-1/(d+2)}$ and $m \geq  n^{d/(d+2)}$, which yields an upper bound of order $n^{-1/(d+2)}$ up to some logarithmic terms.
Note that the previous choice of $m $ and $\eta$ automatically satisfies the condition of Theorem \ref{th:main_th_OptiNet} when $n$ is large enough. 

Let us take \(1/\delta = n \log(n)^2\) and \(\eta = (\log(n)/n)^{1/(d+2)}\). Choose $m$ at least larger than \(n^{d/(d+2)}\) so that $\log(m/\delta)/{m} \leq \eta^d$ (this is ensured as soon as \(m \gtrsim n^{d/(d+2)} \log(n)^{2/(d+2)}\)). In this way, the condition on $m$ in Theorem \ref{th:main_th_OptiNet} is satisfied for large enough $n$, and by the Borel–Cantelli Lemma, we obtain that for each \(x \in S_X\), almost surely  
$$|\hat g_{\textrm{opt}} (x) - g(x) | \leq C \,  \left(\frac{ \log(n)}{n}  \right)^{1/(d+2)}$$  
where \(C > 0\) is a constant depending on all the problem parameters, but independent of $n$.

The underlying partition $\mathcal{V}_\eta$ of the OptiNet algorithm  satisfies the $\gamma$-shape regularity in probability, in the sense that there exists a constant \(c > 0\) such that, with probability at least \(1 - \delta\),  
$$\diam(\mathcal V_\eta(x))^d  \leq c \lambda(\mathcal V_\eta(x))\;,$$
whenever $\log(m/\delta)/m \leq \eta^d.$
In particular, for the above choices of $\delta, \eta$, and $m$, we have that, almost surely, for large enough $n$, the cell \(\mathcal V_\eta(x)\) constructed by the OptiNet algorithm is $\gamma$ shape-regular with $\gamma = c$.

\bibliography{b2}

@book{devroye96probabilistic,
  title={A probabilistic theory of pattern recognition},
  author={Devroye, Luc and Gy{\"o}rfi, L{\'a}szl{\'o} and Lugosi, G{\'a}bor},
  volume={31},
  year={1996},
  publisher={Springer Science \& Business Media}
}

@article{lakshminarayanan2014mondrian,
  title={Mondrian forests: Efficient online random forests},
  author={Lakshminarayanan, Balaji and Roy, Daniel M and Teh, Yee Whye},
  journal={Advances in neural information processing systems},
  volume={27},
  year={2014}
}

@inproceedings{kerem2023error,
  title={On error and compression rates for prototype rules},
  author={Kerem, Omer and Weiss, Roi},
  booktitle={Proceedings of the AAAI Conference on Artificial Intelligence},
  volume={37},
  pages={8228--8236},
  year={2023}
}

@article{portier2021nearest,
  title={Nearest neighbor process: weak convergence and non-asymptotic bound},
  author={Portier, Fran{\c{c}}ois},
  journal={arXiv preprint arXiv:2110.15083},
  year={2021}
}

@article{arlot2014analysis,
  title={Analysis of purely random forests bias},
  author={Arlot, Sylvain and Genuer, Robin},
  journal={arXiv preprint arXiv:1407.3939},
  year={2014},
    number ={},
    volume={},
    pages={}
}

@techreport{breiman2000some,
  title={Some infinity theory for predictor ensembles},
  author={Breiman, Leo},
  year={2000},
  institution={Citeseer}
}

@book {boucheron2013concentration,
    AUTHOR = {Boucheron, St\'{e}phane and Lugosi, G\'{a}bor and Massart,
              Pascal},
     TITLE = {Concentration inequalities. A nonasymptotic theory of independence},
       PUBLISHER = {Oxford University Press, Oxford},
      YEAR = {2013},
     PAGES = {x+481},
      ISBN = {978-0-19-953525-5},
   MRCLASS = {60E15 (60D05 60G15 60G50 60G70 62G20)},
}

@book{anderson,
  title={Some nonparametric multivariate procedures based on statistically equivalent blocks},
  author={Anderson, T. W.},
  publisher={Multivariate Analysis (P. R. Krishnaiah, ed), 5-27, Academic Press, New York.},
  year={1966}
}

@inproceedings{NIPS2014_8c19f571,
 author = {Gottlieb, Lee-Ad and Kontorovich, Aryeh and Nisnevitch, Pinhas},
 booktitle = {Advances in Neural Information Processing Systems},
 editor = {Z. Ghahramani and M. Welling and C. Cortes and N. Lawrence and K.Q. Weinberger},
 pages = {},
 publisher = {Curran Associates, Inc.},
 title = {Near-optimal sample compression for nearest neighbors},
 volume = {27},
 year = {2014}
}

@inproceedings{NIPS2017_934815ad,
 author = {Kontorovich, Aryeh and Sabato, Sivan and Weiss, Roi},
 booktitle = {Advances in Neural Information Processing Systems},
 editor = {I. Guyon and U. Von Luxburg and S. Bengio and H. Wallach and R. Fergus and S. Vishwanathan and R. Garnett},
 pages = {},
 publisher = {Curran Associates, Inc.},
 title = {Nearest-Neighbor Sample Compression: Efficiency, Consistency, Infinite Dimensions},
 url = {https://proceedings.neurips.cc/paper_files/paper/2017/file/934815ad542a4a7c5e8a2dfa04fea9f5-Paper.pdf},
 volume = {30},
 year = {2017}
}

@article{BPS,
  title={Revisiting local regression: shape regularity, uniform rates, and the limits of random splits},
  author={Bettinger, Jérémy and Portier, François and Saumard, Adrien},
  journal={arXiv preprint arXiv:2606.28641},
  volume={},
  number={},
  pages={},
  year={2026},
  publisher={}
}

@misc{xue2018achieving,
      title={Achieving the time of $1$-NN, but the accuracy of $k$-NN}, 
      author={Lirong Xue and Samory Kpotufe},
      year={2017},
      eprint={1712.02369},
      archivePrefix={arXiv},
      primaryClass={math.ST},
      url={https://arxiv.org/abs/1712.02369}, 
}

@article{cerou2006nearest,
  title={Nearest neighbor classification in infinite dimension},
  author={Frédéric Cérou and Arnaud Guyader},
  journal={ESAIM: Probability and Statistics.},
  year={2006},
  DOI= "10.1051/ps:2006014",
  url= "https://doi.org/10.1051/ps:2006014",
}

@inproceedings{mondrianroy2008,
  title={The {M}ondrian process},
  author={Roy, Daniel M and Teh, Yee W},
  booktitle={Advances in {N}eural {I}nformation {P}rocessing {S}ystems},
  pages={1377--1384},
  year={2008}
}

@book{breiman1984classification,
  title={Classification and Regression Trees},
  author={Breiman, Leo and Friedman, Jerome and Stone, Charles J and Olshen, RA},
  year={1984},
  publisher={CRC Press}
}

@article{mondrian,
  title={Universal consistency and minimax rates for online Mondrian Forests},
  author={Jaouad Mourtada and Stéphane Gaïffas and Erwan Scornet},
  journal={Advances in Neural Information Processing Systems 30},
  year={2017},
  publisher={}
}

@article{minimaxmondrian,
  title={Minimax optimal rates for Mondrian trees and forests},
  author={Jaouad Mourtada and Stéphane Gaïffas and Erwan Scornet},
  journal={Annals of Statistics},
    number ={},
    volume={},
    pages={},
  year={2019}
}

@inproceedings{klusowski2021sharp,
  title={Sharp analysis of a simple model for random forests},
  author={Klusowski, Jason},
  booktitle={International Conference on Artificial Intelligence and Statistics},
  pages={757--765},
  year={2021},
  organization={PMLR}
}

@article{hanneke2021universal,
  title={Universal Bayes consistency in metric spaces},
  author={Hanneke, Steve and Kontorovich, Aryeh and Sabato, Sivan and Weiss, Roi},
  journal={The Annals of Statistics},
  volume={49},
  number={4},
  pages={2129--2150},
  year={2021},
  publisher={Institute of Mathematical Statistics}
}

@article{gyorfi2021universal,
  title={Universal consistency and rates of convergence of multiclass prototype algorithms in metric spaces},
  author={Gy{\"o}rfi, L{\'a}szl{\'o} and Weiss, Roi},
  journal={Journal of Machine Learning Research},
  volume={22},
  number={151},
  pages={1--25},
  year={2021}
}

@article{biau2012analysis,
  title={Analysis of a random forests model},
  author={Biau, G{\'e}rard},
  journal={The Journal of Machine Learning Research},
  volume={13},
  number={1},
  pages={1063--1095},
  year={2012},
  publisher={JMLR. org}
}

@article{biau2016random,
  title={A random forest guided tour},
  author={Biau, G{\'e}rard and Scornet, Erwan},
  journal={Test},
  volume={25},
  pages={197--227},
  year={2016},
  publisher={Springer}
}

\newpage 

\section*{Mathematical proofs} 
Let $\PP$ be the probability measure on the underlying probability space $(\Omega, \mathcal F)$ on which are defined all introduced random variables.

\subsection*{Proof of Theorem \ref{sousgauss3}}
Let \(x \in S_X\). Set $\varepsilon_i := Y_i - g(X_i)$ for each $i=1,\ldots, n$.  We write the bias-variance decomposition 
$\hat g_{\mathcal V}(x) - g(x) = W + B$, where \begin{align*}
    W := \dfrac{\sum_{i=1}^n \varepsilon_i \mathds 1 _{ \mathcal V(x) }(X_i)}{\sum_{j=1}^n \mathds 1 _{ \mathcal V(x) }(X_j)} \qquad \text{and}\qquad
       B := \dfrac{\sum_{i=1}^n \left(g(X_i) - g(x)\right) \mathds 1 _{ \mathcal V(x) }(X_i)}{\sum_{j=1}^n \mathds 1 _{ \mathcal V(x) }(X_j)}.
\end{align*}

Let us revisit the idea of the proof of Theorem 6 in \cite{BPS}, here, however we are not dealing with a uniform version. For all $i \in \{1, \dots, n\}$, let us denote $g_i = \ind_{\mathcal{V}(x)}(X_i)$ and \( \mathbb{P}_{X_{1:n}} \) the probability \( \mathbb{P} \) conditional on \( {X}_1, \dots, {X_n} \). 
Since the conditional distribution of $\varepsilon_i$ given $X_1,\ldots,X_n$ is sub-Gaussian with parameter $\sigma^2$, then  $\varepsilon_i g_i$ is sub-Gaussian under $\pr_{X_{1:n}}$, with parameter $\sigma^2 g_i^2$. Hence, by conditional independence given $(X_i)_{i = 1,\ldots, n}$, ${\sum_{i=1}^n \varepsilon_i g_i}/{\sqrt{ \sum_{j=1}^n g_j}}$ is sub-Gaussian with parameter $\sigma^2 \sum_{i=1}^n g_i^2 / \sum_{j=1}^n g_j$. Indeed, $(Y_i) _{i = 1,\ldots, n} $ (and so $(\varepsilon_i) _{i = 1,\ldots, n} $) is an independent collection of random variables, conditionally on $(X_i)_{i = 1,\ldots, n}$. We prove this fact in Lemma \ref{lemmaA}.  Moreover, $\sum_{i=1}^n g_i^2 = \sum_{i=1}^n g_i$ because $g_i\in \{ 0,1\} $. Hence, ${\sum_{i=1}^n \varepsilon_i g_i}/{\sqrt{ \sum_{j=1}^n g_j}}$ is sub-Gaussian with parameter $\sigma^2$ under $\pr_{X_{1:n}}$. It follows that
\begin{align*}
    \mathbb{P}_{X_{1:n}}\left(\dfrac{\sum_{i=1}^n \varepsilon_i \ind_{ {\mathcal{V}(x)} }(X_i)}{\sqrt{\sum_{j=1}^n \ind_{ {\mathcal{V}(x)} }(X_j)}} > t  \right)  \leq \exp\left( \dfrac{-t^2}{2\sigma^2}\right) = \delta,
\end{align*}
   with  $t = \sqrt{2\sigma^2 \log(1/\delta)}$.
 Integrating with respect to $X_1,\dots, X_n$, we obtain the same inequality with $\mathbb P $ instead of $\mathbb{P}_{X_{1:n}}$. By symmetry, we obtain the result with absolute values with probability at least \( 1 - 2\delta \). 
 We have shown that with probability at least $1-2\delta$,
\begin{align}\label{eq1}
    \left| \dfrac{ \sum_{i=1}^n \varepsilon_i \ind_{{\mathcal{V}(x)}}(X_i) } {\sqrt{\sum_{j=1}^n \ind_{ {\mathcal{V}(x)} }(X_j)}}  \right| \leq  \sqrt{2\sigma^2 \log(1/\delta)}.
\end{align}
Then with probability at least $1-2\delta$ we have, 
$$|W| = \left| \dfrac{ \sum_{i=1}^n \varepsilon_i \ind_{{\mathcal{V}(x)}}(X_i) } {{\sum_{j=1}^n \ind_{ {\mathcal{V}(x)} }(X_j)}}  \right| \leq  \sqrt{\dfrac{2\sigma^2 \log(1/\delta)}{nP_n^X({\mathcal{V}(x)})}}.$$

Furthermore, using the triangle inequality, we obtain that
\begin{eqnarray*}
    |B| &\leq& \dfrac{\sum_{i=1}^n \left|g(X_i) - g(x)\right| \mathds 1 _{ \mathcal V(x) }(X_i)}{\sum_{j=1}^n \mathds 1 _{ \mathcal V(x) }(X_j)} \\ &\leq& \dfrac{\sum_{i=1}^n \sup_{y \in \mathcal V(x)} |g(y) - g(x)| \mathds 1 _{ \mathcal V(x) }(X_i)}{\sum_{j=1}^n \mathds 1 _{ \mathcal V(x) }(X_j)} = \sup_{y \in \mathcal V(x)} |g(y) - g(x)|.
\end{eqnarray*}
Moreover, using the Lipschitz assumption, it follows that
$$|g(y) - g(x)| \leq  L(\mathcal V(x)) \|x-y\|_2  \leq L(\mathcal V(x)) \diam  (\mathcal V(x) ) , $$
which concludes the proof. Note that the proof of result \eqref{eq1} remains valid for $\mathcal{V}$ created by variables $\xi_i$ independent of the $X_i$. Indeed, it suffices to reason conditionally on these variables $\xi_i$ and then integrate with respect to $\xi_i$ and $X_i$.

\subsection*{Proof of Corollary \ref{sousgauss}}
According to Theorem \ref{sousgauss3}, it now remains to show that, for the measurable set $\mathcal{V}(x)$, we have with probability at least $1-\delta$,
\begin{align}\label{eq2}
   \sum_{j=1}^n \ind_{ \mathcal{V}(x) }(X_j) = n 
P_n^X(\mathcal{V}(x))  \geq  \dfrac{n P^X(\mathcal{V}(x))}{2}.
\end{align}
Indeed, it can easily be seen that \eqref{eq1} and \eqref{eq2} imply the stated inequality and these inequalities hold together with probability at least $1-3\delta$. 

Define $W_i = \ind_{ \mathcal{V}(x) } (X_i) $. Note that $W_1,\ldots, W_n$ is an independent and identically distributed collection of Bernoulli variables with parameter $\mu = P^X (\mathcal{V}(x))  $.
 We have the following inequality for any \( \theta \in (0,1) \)
\[
\PP \left(\sum_{i=1}^n W_i \leq (1 - \theta) n \mu  \right) \leq e^{-\theta^2 n \mu /2}.
\]  
Furthermore, for any \( \delta \in (0,1) \), we have  
\[
\PP  \left( \frac{1}{n} \sum_{i=1}^n W_i \leq \left[1 - \sqrt{\frac{2\log(1/\delta)}{ n  \mu }} \right]\mu  \right) \leq \delta.
\]  
Since $n\mu  \geq 8 \log(1/\delta)$, we obtain with probability at least $1-\delta$,
$ \sum_{i=1}^n W_i >  n \mu / 2$ which yields \eqref{eq2} and concludes the proof. Note that we have established the following: Under \ref{cond:D} and \ref{cond:epsilon}, if $ V  $ is a measurable set such that $ nP^X(V )  \geq 8 \log(1/\delta) $, we have, with probability at least $1-3\delta$,
\begin{align}\label{eq10}
|W| = \left| \dfrac{ \sum_{i=1}^n \varepsilon_i \ind_{V}(X_i) } {{\sum_{j=1}^n \ind_{ V }(X_j)}}  \right| \leq  \sqrt{\dfrac{4\sigma^2 \log(1/\delta)}{nP^X(V)}}.    
\end{align}
It is worth noting that, similarly, this result generalizes to a partition created by a source independent of that of the $X_i$ by conditioning.

\subsection*{Proof of Proposition \ref{prop:diam_CRT}}

First notice that, by a union bound and symmetry in the directions, we have \[ \pr(\diam (\mathcal{V}(x))\geq t) \leq d\pr \left(h_1 \geq \frac{t}{\sqrt{d}}\right). \]
Furthermore, by denoting $B_i^{(1)}=\ind_{D_i=1}$, we get for any $r\in (0,1)$ and $\lambda>0$,
\begin{align*}
    \pr(h_1 \geq r^N) & = \pr\left(\prod_{i=1}^N 2^{-B_i^{(1)}} \geq r^N\right)\\
    & \leq \mathbb E\left[\left(\frac{\prod_{i=1}^N 2^{-B_i^{(1)}}}{r^N}\right)^\lambda \right]=\left(\frac{\mathbb E\left[{2^{-\lambda B_1^{(1)}}}\right]}{r^{\lambda}}\right)^N. 
\end{align*}
It holds
\[
 \mathbb E\left[{2^{-\lambda B_1^{(1)}}}\right]= \frac{1}{d2^{\lambda}}+1-\frac{1}{d}.
\]
Hence,
\[
 \pr(h_1 \geq r^N) \leq \left(\frac{1}{d2^{\lambda}}+1-\frac{1}{d}\right)^N r^{-\lambda N}.
\]
Let us set $r=2^{-\alpha}$ and define 
\[
h(\lambda)= Q(\lambda) 2^{\lambda \alpha}
\]
with
\[
Q(\lambda)= \frac{1}{d2^{\lambda}}+1-\frac{1}{d}.
\]
By differentiating in $\lambda$, we get
\[
h^{\prime}(\lambda)=\log(2) 2^{\lambda \alpha}\left(\alpha Q(\lambda)-\frac{1}{d2^\lambda}\right).
\]
Hence, $h^{\prime}(\lambda_0)=0$ for $\lambda_0$ such that $2^{-\lambda_0}=\theta=\alpha (d-1)/(1-\alpha)$ and $\alpha \in (0,1/d)$. With this choice of $\lambda$, 
\[
 \PP(\diam(\mathcal{V}(x))\geq \sqrt{d}2^{-\alpha N}) \leq d\left(1-\frac{1-\theta}{d}\right)^N \theta^{-\alpha N}.
\]
We proceed in the same way as before for the diameter upper bound. By a union bound and symmetry in the directions, we have
\[
 \pr(  \diam (\mathcal{V}(x))\leq t) \leq d\pr \left(h_1 \leq \frac{t}{\sqrt{d}}\right).
\]
Then, for any $r\in (0,1)$ and $\lambda>0$,
\begin{align*}
    \pr(h_1 \leq r^N) & = \pr\left(\prod_{i=1}^N 2^{B_i^{(1)}} \geq r^{-N}\right)\\
    & \leq \mathbb E\left[\left(\frac{\prod_{i=1}^N 2^{B_i^{(1)}}}{r^{-N}}\right)^\lambda \right]=\left(\frac{\mathbb E\left[{2^{\lambda B_1^{(1)}}}\right]}{r^{-\lambda}}\right)^N. 
\end{align*}
It holds
\[
 \mathbb E\left[{2^{\lambda B_1^{(1)}}}\right]= \frac{2^{\lambda}}{d}+1-\frac{1}{d}.
\]
Hence,
\[
 \pr(h_1 \leq r^N) \leq \left(\frac{2^{\lambda}}{d}+1-\frac{1}{d}\right)^N r^{\lambda N}.
\]
Let us set $r=2^{-\alpha}$ and denote 
\[
h(\lambda)= Q(\lambda) 2^{-\lambda \alpha}
\]
with
\[
Q(\lambda)= \frac{2^{\lambda}}{d}+1-\frac{1}{d}.
\]
By differentiating in $\lambda$, we get
\[
h^{\prime}(\lambda)=\log(2) 2^{-\lambda \alpha}\left(\frac{2^\lambda}{d}-\alpha Q(\lambda)\right).
\]
Hence, $h^{\prime}(\lambda_0)=0$ for $\lambda_0$ such that $2^{\lambda_0}=\theta=\alpha (d-1)/(1-\alpha)$ and $\alpha \in (1/d , 1)$. With this choice of $\lambda$, 
\[
 \PP(\diam(\mathcal{V}(x))\leq \sqrt{d}2^{-\alpha N}) \leq d\left(1-\frac{1-\theta}{d}\right)^N \theta^{-\alpha N}.
\]
\qed

\subsection*{Proof of Proposition \ref{cor22}}
According to Proposition \ref{prop:diam_CRT}, for any $\alpha \in (1/d,1)$, we have, for $\theta = \alpha (d-1)/(1-\alpha)$,
\begin{equation*}
    \PP(\diam(\mathcal{V}(x))\leq \sqrt{d}2^{-\alpha N}) \leq d\left(1-\frac{1-\theta}{d}\right)^N \theta^{-\alpha N}.
\end{equation*}
Take now $\alpha=\alpha_N=1/d+\omega_N$, with $\omega_N \rightarrow_{N\rightarrow +\infty} 0$. In this case,
\[
\theta=\theta_N=\frac{\alpha_N(d-1)}{1-\alpha_N}=(d-1)\frac{1+d\omega_N}{d-1-d\omega_N}=1+a_d \omega_N+ b_d \omega_N^2+O(\omega_N^3),
\]
where $a_d=d^2/(d-1)$ and $b_d=d^3/(d-1)^2$. This gives
\[
\log\left(1-\frac{1-\theta}{d}\right)= \frac{a_d}{d} \omega_N+ \frac{b_d}{d}\omega_N^2 - \frac{a^2_d}{2d^2}\omega_N^2 + O(\omega_N^3).
\]
Furthermore $$\log(\theta) = a_d \omega_N + b_d \omega_N^2 -  \frac{a_d^2}{2} \omega_N^2 + O(\omega_N^3)$$
so $$\alpha \log(\theta) = \frac{a_d}{d} \omega_N + \frac{b_d}{d} \omega_N^2 -  \frac{a_d^2}{2d} \omega_N^2 + a_d \omega_N^2 + O(\omega_N^3).$$
Then 
\begin{eqnarray*}
    \log\left(1-\frac{1-\theta}{d}\right) - \alpha \log(\theta) &=& -\frac{a^2_d}{2d^2}\omega_N^2 + \frac{a_d^2}{2d} \omega_N^2 - a_d \omega_N^2 + O(\omega_N^3) \\ &=& - a_d \omega_N^2 \left( 1 + \frac{a_d}{2d^2} - \frac{a_d}{2d} \right) + O(\omega_N^3).
\end{eqnarray*}
Moreover 
$$1 + \frac{a_d}{2d^2} - \frac{a_d}{2d}  = 1 + \frac{1}{2(d-1)} - \frac{d}{2(d-1)} = 1 - \frac{1}{2} = \frac{1}{2}.$$
Finally
\begin{eqnarray*}
  \left(1-\frac{1-\theta}{d}\right)^N \theta^{-\alpha N} &=& \exp \left( N \log\left(1-\frac{1-\theta}{d}\right) - N \alpha \log(\theta)    \right) \\ &=& \exp \left( - \frac{a_d}{2} N \omega_N^2 + O(N \omega_N^3)  \right).  
\end{eqnarray*}
Choosing $\omega_N=2 \sqrt{\log(N)/(a_d N)} \in (0 , 1 - 1/d)$ for $N$ large enough, gives
$$\left(1-\frac{1-\theta}{d}\right)^N \theta^{-\alpha N} = \exp \left( - 2 \log(N) + O \left(\log(N)^{3/2} / \sqrt{N}   \right) \right) \underset{ {N \to + \infty}}\sim N^{-2}$$
and concludes the proof via the Borel-Cantelli Lemma. Furthermore, for the upper bound of the diameter, we also use Proposition \ref{prop:diam_CRT}. For any $\alpha \in (0,1/d)$, we have for $\theta = \alpha (d-1)/(1-\alpha)$,
\[
 \PP(\diam(\mathcal{V}(x))\geq \sqrt{d}2^{-\alpha N}) \leq d\left(1-\frac{1-\theta}{d}\right)^N \theta^{-\alpha N}.
\]
Let us take here $\alpha=\alpha_N=1/d-\omega_N$, with $\omega_N \rightarrow_{N\rightarrow +\infty} 0$. In this case,
\[
\theta=\theta_N=\frac{\alpha_N(d-1)}{1-\alpha_N}=(d-1)\frac{1-d\omega_N}{d-1+d\omega_N}=1-a_d \omega_N+ b_d \omega_N^2+O(\omega_N^3),
\]
where $a_d=d^2/(d-1)$ and $b_d=d^3/(d-1)^2$. This gives
\[
\log\left(1-\frac{1-\theta}{d}\right)= - \frac{a_d}{d} \omega_N+ \frac{b_d}{d}\omega_N^2 - \frac{a^2_d}{2d^2}\omega_N^2 + O(\omega_N^3).
\]
In addition, $$\log(\theta) = - a_d \omega_N + b_d \omega_N^2 -  \frac{a_d^2}{2} \omega_N^2 + O(\omega_N^3)\;,$$
so $$\alpha \log(\theta) = - \frac{a_d}{d} \omega_N + \frac{b_d}{d} \omega_N^2 -  \frac{a_d^2}{2d} \omega_N^2 + a_d \omega_N^2 + O(\omega_N^3)\;.$$
Now,
\begin{eqnarray*}
    \log\left(1-\frac{1-\theta}{d}\right) - \alpha \log(\theta) &=& -\frac{a^2_d}{2d^2}\omega_N^2 + \frac{a_d^2}{2d} \omega_N^2 - a_d \omega_N^2 + O(\omega_N^3) \\ &=& - a_d \omega_N^2 \left( 1 + \frac{a_d}{2d^2} - \frac{a_d}{2d} \right) + O(\omega_N^3).
\end{eqnarray*}
Moreover,
$$1 + \frac{a_d}{2d^2} - \frac{a_d}{2d}  = 1 + \frac{1}{2(d-1)} - \frac{d}{2(d-1)} = 1 - \frac{1}{2} = \frac{1}{2}.$$
Finally,
\begin{eqnarray*}
  \left(1-\frac{1-\theta}{d}\right)^N \theta^{-\alpha N} &=& \exp \left( N \log\left(1-\frac{1-\theta}{d}\right) - N \alpha \log(\theta)    \right) \\ &=& \exp \left( - \frac{a_d}{2} N \omega_N^2 + O(N \omega_N^3)  \right).  
\end{eqnarray*}
Choosing $\omega_N=2 \sqrt{\log(N)/(a_d N)}$ gives
$$\left(1-\frac{1-\theta}{d}\right)^N \theta^{-\alpha N} = \exp \left( - 2 \log(N) + O \left(\log(N)^{3/2} / \sqrt{N}   \right) \right) \underset{ {N \to + \infty}}\sim N^{-2}$$
and concludes the proof via the Borel-Cantelli Lemma.

The last inequality follows directly by invoking the two previous inequalities on diameter and volume. \qed

\subsection*{Proof of Theorem \ref{cor_UcR}}
We can apply Corollary \ref{sousgauss} pointwise for $x \in S_X$ and, with $\delta = n^{-2}$, we find that, whenever $n P^X (\mathcal{V}(x) ) / \log(n) \geq 16$, it holds that
$$ \sum_{n\geq 1} \mathbb P ( |\hat g_{\mathcal V}(x) - g(x)| > v_n ) < \infty,$$
where $$v_n = \sqrt{ 8 \sigma^2  \log( n  ) / (n  b \lambda ( \mathcal V(x)) )}  + L(\mathcal V(x)) \diam (\mathcal V(x))$$
since $P^X(\mathcal V(x)) \geq b \lambda(\mathcal V(x))$ by \ref{cond:density_XCART}.

Applying the Borel-Cantelli Lemma, we get that with probability $1$, for $n$ large enough,
$$|\hat g_{\mathcal V}(x) - g(x)| \leq  \sqrt{\frac{ 8 \sigma^2  \log( n)}{n  b \lambda ( \mathcal V(x)) }} + L(\mathcal V(x)) \diam (\mathcal V(x)).$$ Then, using \ref{cond:density_XCART}, it follows that $$ P^X (\mathcal{V}(x) ) \geq b \lambda(\mathcal{V}(x)) = b 2^{-N} = n^{-d/(d+2)} b $$ so $$ n P^X (\mathcal{V}(x) ) \geq n^{2/(d+2)} b.$$ Hence, we get that $ n P^X (\mathcal{V}(x) ) / \log(n) \geq 16$ for $n$ sufficiently large.

Then, we have the following inequality, with probability $1$, for $n $ large enough,
$$|\hat g_{\mathcal V}(x) - g(x)| \leq  \sqrt{\frac{ 8 \sigma^2  \log\left( n \right)}{n  b \lambda ( \mathcal V(x)) }} + L(\mathcal V(x)) \diam (\mathcal V(x)).$$
Now, from Proposition \ref{cor22}, for a sufficiently large $N$, we have
\begin{align*}
& \diam (\mathcal{V}(x)) \leq \sqrt d 2^{-N/d+2\sqrt{(d-1)N\log(N)/d^2}}, \\
& \lambda(\mathcal{V}(x)) = 2^{-N}.
\end{align*}
Hence, we get, with probability $1$, for $n$ large enough,
$$| \hat g_{\mathcal V}(x) - g(x)| \leq \sqrt{\frac{ 8 \sigma^2  \log\left( n \right)}{n b 2^{-N}}} + L(\mathcal V(x)) \sqrt d 2^{-N/d+2\sqrt{(d-1)N\log(N)/d^2}}.$$
Because \( N = d\log(n)/(\log(2)(d+2)) \), we obtain
\begin{align*}
    | \hat g_{\mathcal V}(x) - g(x)| &\leq n^{-1/(d+2)} \sqrt{\frac{ 8 \sigma^2  \log\left( n \right)}{b}} + n^{-1/(d+2)} L(\mathcal V(x)) \sqrt d e^{2\sqrt{(d-1)N\log(N)/d^2}}\\
& \leq n^{-1/(d+2)} \sqrt{\frac{ 8 \sigma^2  \log\left( n \right)}{b}} + n^{-1/(d+2)} L(\mathcal V(x)) \sqrt d e^{\sqrt{N\log(N)}}
\end{align*}
where we use the inequality $2 \sqrt{(d-1)/d^2} \leq 1$ since $(d-2)^2 \geq 0$.
For $n$ large enough, we have \( N = d\log(n)/(\log(2)(d+2)) \geq 8.\) Thus, this implies that $\log(n) = N \log(2) (d+2)/ d \leq 3 \log(2) N = \log(8) N \leq \log(N) N$. 
Moreover $N \leq 2 \log(n)$, and for $n$ large enough $N \leq \log(n)^2$ then $\log(N) \leq 2 \log \log(n).$ Finally $\log(n) \leq \log(N) N \leq  4 \log(n) \log(\log(n)).$ We conclude by using the inequality $\sqrt{x} \leq e^{\sqrt{x}}$ for $x= \log(n)$ and setting $C_x = \sqrt{{8 \sigma^2 }/b} +  L(\mathcal V(x)) \sqrt d \leq \sqrt{{8 \sigma^2 }/b} +  L \sqrt d = C$. \qed

\subsection*{Proof of Proposition \ref{centered tree not regular}}
Let $d \geq 2$. At each stage, for each terminal leaf, draw uniformly $D_i $ in $\{1,\ldots, d\}$ and split at the midpoint, i.e. $S_i = 1/2$. Then we divide the cell according to coordinate $k = D_i$. The corresponding length $h_k(\mathcal V (x) )$ is then updated into $h_k(\mathcal V (x) )/2$. As a consequence, for a given leaf, after $N$ stages,  the $k$-th length has the following representation 
$$ h _ k(\mathcal V (x) ) = 2^{-B_1^{(k)}}\times \ldots \times 2^{-B_N^{(k)}} = \exp\left( \sum_{i=1} ^ N B_i^{(k)} \log(1/2) \right)  $$
where $B_i^{(k)}  = \ind_{D_i = k }$. It follows that
\begin{align*}
 &h_+(\mathcal V (x) ) = \exp\left( \max_{k=1,\ldots, d}  \sum_{i=1} ^ N B_i^{(k)} \log(1/2) \right), \\
 &h_-(\mathcal V (x) ) =  \exp\left(\min_{k=1,\ldots, d}  \sum_{i=1} ^ N B_i^{(k)} \log(1/2) \right), 
\end{align*}
and the expression of the ratio is
\begin{align*}
 h_+(\mathcal V (x) )/h_-(\mathcal V (x) )  &= \exp \left(  \log(2) \max_{1 \leq k,j \leq d} \sum_{i=1} ^N (B_i^{(k)} - B_i^{(j)} ) \right). \end{align*}
By denoting $V_i^{k,j} = B_i^{(k)} - B_i^{(j)}$, we get
\[V_i^{k,j} = \begin{cases} 
1 & \text{with probability } 1/d \\
0 & \text{with probability } 1 - 2/d \\
-1 & \text{with probability } 1/d 
\end{cases}.\]
Note that the variables \((V_i^{k,j})_{i=1}^N\) are mutually independent because the \((D_i)_{i=1}^N\) are independent. 
Let \(Z_{k,j} = \sum_{i=1}^{N} V_i^{k,j} \log(2)\) such that
\[
\frac{h_+(\mathcal V (x) )}{h_-(\mathcal V (x) )} = \exp \left( \max_{1 \leq k,j \leq d} Z_{k,j} \right).
\]
Note that \(Z_{k,j} = - Z_{j,k}\) and \(Z_{k,k} = 0\), which gives
\[
\max_{1 \leq k,j \leq d} Z_{k,j} = \max_{1 \leq k < j \leq d} | Z_{k,j} |
\]
and thus the formula 
\[
\frac{h_+(\mathcal V (x) )}{h_-(\mathcal V (x) )} = \exp \left( \max_{1 \leq k < j \leq d} | Z_{k,j} | \right).
\]
By using the Paley-Zygmund inequality to $Z_{k,j}^2$, we get for all $\theta \in (0,1)$,
\[
\PP\left(|Z_{k,j}| \geq \sqrt{\theta} \sqrt{\mathbb{E}(Z_{k,j}^2)} \right) \geq (1-\theta)^2 \frac{\mathbb{E}(Z_{k,j}^2)^2}{\mathbb{E}(Z_{k,j}^4)}.
\]
We therefore seek to calculate the 2nd and 4th moments of \(Z_{k,j}\). \\ Since \(Z_{k,j}^2 = \sum_{i \neq \ell} V_i^{k,j} V_\ell^{k,j} \log(2)^2 + \sum_{i=1}^N V_i^{k,j \, 2} \log(2)^2 \), \(\mathbb{E}(V_i^{k,j}) = 0\), we obtain 
$$\mathbb E(Z_{k,j}^2) = \sum_{i=1}^N \mathbb{E}((V_i^{k,j})^2) \log(2)^2 = N \times \frac{2}{d} \times \log(2)^2 = \frac{2N}{d} \log(2)^2.$$

Moreover, we obtain with Lemma \ref{lemme moment},
\begin{eqnarray*}
    \mathbb{E}(Z_{k,j}^4) &=& N \mathbb{E}({V_i^{k,j \, 4}})\log(2)^4 + 3N (N-1) \mathbb{E}({V_i^{k,j \, 2}})^2 \log(2)^4 \\ &=& N \times \dfrac{2}{d} \times  \log(2)^4 + 3N(N-1) \left( \dfrac{2}{d}\right)^2 \times \log(2)^4 \\ &=& \dfrac{2N}{d^2} \log(2)^4 (6N - 6 + d).
\end{eqnarray*}
The Paley-Zygmund bound becomes
\[
\frac{\mathbb{E}(Z_{k,j}^2)^2}{\mathbb{E}(Z_{k,j}^4)} = \frac{4 \log(2)^4 N^2}{d^2} \times \frac{d^2}{2N \log(2)^4 (6N - 6 + d) }= \frac{2N}{6N - 6 + d}.
\]
Thus, for all \(\theta \in (0,1)\) and $N \geq d$,
\[
\PP\left(|Z_{k,j}| \geq \sqrt{\theta} \sqrt{\frac{2N \log(2)^2}{d}}\right) \geq (1-\theta)^2 \frac{2N}{6N - 6 + d} \geq \dfrac{2(1-\theta)^2}{7}.
\]
Let us choose \(\theta = 1/2\) to obtain
\[
\PP\left(|Z_{k,j}| \geq \log(2) \sqrt{\frac{N}{d}}\right) \geq \dfrac{1}{14}
\]
and thus for $N \geq d$,
\[
\PP\left(\dfrac{h_+(\mathcal V (x))}{h_-(\mathcal V (x))} \geq 2^{\sqrt{{N}/{d}}} \right) \geq \frac{1}{14}.
\]
Thus, with probability at least $1/14$, the ratio ${h_+(\mathcal V (x))}/{h_-(\mathcal V (x))} $ is bounded below by a quantity that grows exponentially towards infinity. This means that centered trees are not shape regular. \qed

\subsection*{Proof of Proposition \ref{prop:diamURT2}}

At each split step $i$, the exact proportion $\bar{S}_i$ by which the side length of the cell $\mathcal{V}(x)$ is reduced takes either the value $S_i$ or $1-S_i$, depending on whether the target point $x$ falls into the left or the right child cell. Then, we can almost surely bound this reduction from both above and below at each step
$$ m_i = \min(S_i, 1-S_i) \leq \bar{S}_i \leq \max(S_i, 1-S_i) = M_i. $$

Since the original split variables $(S_i)_{i \geq 1}$ are independent and identically distributed, it follows by construction that $(m_i)_{i \geq 1}$ and $(M_i)_{i \geq 1}$ are two sequences of mutually independent random variables, uniformly distributed over $[0, 1/2]$ and $[1/2, 1]$ respectively. These independent bounding sequences allow us to securely control the volume and the diameter of the cell $\mathcal{V}(x)$ without requiring explicit knowledge of the point's path through the tree.

Notice that, by a union bound and symmetry in the directions, we have \[ \pr(\diam (\mathcal{V}(x))\geq t) \leq d\pr \left(h_1 \geq \frac{t}{\sqrt{d}}\right). \]
Furthermore, by denoting $B_i^{(1)}=\ind_{D_i=1}$, we get for any $r\in (0,1)$ and $\lambda>0$,
$$    \pr(h_1 \geq r^N) = \pr\left(\prod_{i=1}^N \bar{S}_i^{B_i^{(1)}} \geq r^N\right) \leq \mathbb E\left[\left(\frac{\prod_{i=1}^N \bar{S}_i^{B_i^{(1)}}}{r^N}\right)^\lambda \right] \leq \mathbb E\left[\left(\frac{\prod_{i=1}^N M_i^{B_i^{(1)}}}{r^N}\right)^\lambda \right]. $$
Then, $$ \pr(h_1 \geq r^N) \leq \left(\frac{\mathbb E\left[{M_1^{\lambda B_1^{(1)}}}\right]}{r^{\lambda}}\right)^N$$
It holds
$$\mathbb E\left[{M_1^{\lambda B_1^{(1)}}}\right] = 1 - \frac{1}{d} + \frac{1}{d} \left( \frac{2 - 2^{-\lambda}}{\lambda + 1} \right).$$
Hence,
\[
 \pr(h_1 \geq r^N) \leq \left(1 - \frac{1}{d} + \frac{1}{d} \left( \frac{2 - 2^{-\lambda}}{\lambda + 1} \right)\right)^N r^{-\lambda N}.
\]
First note that it suffices to optimize the bound for $N=1$. Let us denote 
\[
Q(\lambda)= 1 - \frac{1}{d} + \frac{1}{d} \left( \frac{2 - 2^{-\lambda}}{\lambda + 1} \right)
\]
and 
\[
h(\lambda)= Q(\lambda) r^{-\lambda}.
\]
Let $c = 1 - \log 2$. We can rewrite $Q(\lambda)$ as
\[
Q(\lambda) = 1 + \frac{1}{d} \left( \frac{1 - \lambda - 2^{-\lambda}}{\lambda + 1} \right).
\]
Using the convexity inequality $2^{-\lambda} = e^{-\lambda \log 2} \geq 1 - \lambda \log 2$ for all $\lambda \geq 0$, we can bound the numerator: $1 - \lambda - 2^{-\lambda} \leq -\lambda(1 - \log 2) = -c\lambda$. Therefore,
\[
Q(\lambda) \leq 1 - \frac{c\lambda}{d(1+\lambda)}.
\]
Moreover, using the inequality $1-x \leq \exp(-x)$ and $(1+\lambda)^{-1} \geq 1-\lambda$, we obtain $$Q(\lambda)  \leq \exp\left(\frac{-c\lambda}{d(1+\lambda)}\right) \leq \exp\left(\frac{-c\lambda(1-\lambda)}{d}\right).$$
Denote $r=(1/e)^{c/d-\theta}$ for $\theta>0$. Then the function $h(\lambda)= Q(\lambda) r^{-\lambda}$ can be bounded as follows
$$    h(\lambda) \leq  \exp\left(\lambda\left(\frac{c}{d}-\theta\right)-\frac{c\lambda(1-\lambda)}{d}\right) =   \exp\left(-\lambda\theta + \frac{c\lambda^2}{d}\right). $$
By taking $\lambda = {d\theta}/{(2c)}$, we get
\[ \mathbb{P}\left(h_1 \geq e^{N(\theta-c/d)}\right) \leq e^{-d\theta^{2}N / (4c)}.\]
Then for all $\theta > 0$, we obtain
\[
 \pr(\diam (\mathcal V(x))\geq \sqrt{d}e^{N(\theta-c/d)}) \leq d\pr \left(h_1 \geq e^{N(\theta-c/d)}\right) \leq d e^{-d\theta^{2}N/(4c)} \leq d e^{-d\theta^{2}N/2} .
\]
We now consider the lower bound. We proceed in the same way as before. By a union bound and symmetry in the directions, we have
\[
 \pr(\diam (\mathcal V(x))\leq t) \leq d\pr \left(h_1 \leq \frac{t}{\sqrt{d}}\right).
\]
Furthermore, for any $(r, \lambda)\in (0,1)^2$,
$$    \pr(h_1 \leq r^N) = \pr\left(\prod_{i=1}^N \bar{S}_i^{B_i^{(1)}} \leq r^N\right) = \pr\left(\prod_{i=1}^N \bar{S}_i^{-\lambda B_i^{(1)}} \geq r^{-\lambda N} \right) \leq \mathbb E\left[\left(\frac{\prod_{i=1}^N \bar{S}_i^{B_i^{(1)}}}{r^N}\right)^{-\lambda} \right].$$
Then, $$\pr(h_1 \leq r^N) \leq \mathbb E\left[\left(\frac{\prod_{i=1}^N m_i^{B_i^{(1)}}}{r^N}\right)^{-\lambda} \right] =\left(\frac{\mathbb E\left[{m_1^{-\lambda B_1^{(1)}}}\right]}{r^{-\lambda}}\right)^N.$$

It holds,
\[
 \mathbb E\left[{m_1^{-\lambda B_1^{(1)}}}\right] = 1 - \frac{1}{d} + \frac{2^\lambda}{d(1-\lambda)}.
\]
Hence,
\[
 \mathbb{P}(h_1 \leq r^N) \leq \left(1 - \frac{1}{d} + \frac{2^\lambda}{d(1-\lambda)}\right)^N r^{\lambda N}.
\]
Without loss of generality, we can optimize the bound for $N=1$. Let $C = 1 + \log 2$. Define 
$$Q(\lambda) = 1 - \frac{1}{d} + \frac{2^\lambda}{d(1-\lambda)} = 1 + \frac{1}{d} \left( \frac{2^\lambda - 1 + \lambda}{1-\lambda} \right),$$
and
$$ h(\lambda)= Q(\lambda) r^{\lambda}.$$
Using the convex inequality $2^\lambda \leq 1 + \lambda \log 2 + \lambda^2$ for $\lambda \in (0, 1/2)$, the numerator is bounded by $\lambda(1+\log 2) + \lambda^2 = C\lambda + \lambda^2$. Using $(1-\lambda)^{-1} \leq 1+2\lambda$ for $\lambda \in (0,1/2)$, we get
\[
\frac{2^\lambda - 1 + \lambda}{1-\lambda} \leq (C\lambda + \lambda^2)(1+2\lambda) = C\lambda + (2C+1)\lambda^2 + 2\lambda^3.
\]
Since $\lambda \in (0, 1/2)$, we have $2\lambda^3 \leq \lambda^2$. Consequently, the bound simplifies to $C\lambda + (2C+2)\lambda^2$. Knowing that $2C+2 = 4 + 2\log 2 \leq 6$, we obtain $Q(\lambda) \leq 1 + ({C\lambda + 6\lambda^2})/{d} \leq \exp(({C\lambda + 6\lambda^2})/{d})$. Set $r=(1/e)^{C/d+\theta}$ for $\theta>0$, then for all $\lambda \in (0, 1/2)$,
$$    h(\lambda) 
   \leq  \exp\left(-\lambda\left(\frac{C}{d}+\theta\right)+ \frac{C\lambda + 6\lambda^2}{d}\right) =\exp\left(-\lambda\left(\theta-\frac{6\lambda}{d}\right)\right).$$
By taking $\lambda= d\theta/12 \in (0, 1/2)$, we get
\[
 \mathbb{P}\left(h_1 \leq e^{-N(\theta+C/d)}\right) \leq e^{-d\theta^{2}N/24}.
\]
Then for all $\theta \in (0, 6/d)$,
\[
 \pr(\diam(\mathcal V(x))\leq \sqrt{d}e^{-N(\theta+C/d)}) \leq d\pr \left(h_1 \leq e^{-N(\theta+C/d)}\right) \leq d e^{-d\theta^{2}N/24}.
\] \qed

\subsection*{Proof of Proposition \ref{prop:diam_vol_URT2}}
As in the proof of Proposition \ref{prop:diamURT2}, we optimize along some polynomial moments controlling the deviation probability of interest. We have $\lambda(\mathcal V(x))=\prod_{i=1}^N \bar{S}_i$. For any $\alpha > 1$ and $\lambda \in (0,1)$,
\[
\PP\left(\lambda(\mathcal V(x))^{-1}\geq e^{N(\alpha + \log(2))}\right)\leq \EE\left[\prod_{i=1}^N m_i^{-\lambda}\right]e^{-N(\alpha + \log(2)) \lambda} = \left(\frac{2^\lambda e^{-(\alpha + \log(2)) \lambda}}{1-\lambda}\right)^N
\]
where $m_i = \min(S_i, 1-S_i) \sim \mathcal{U}(0,1/2)$. By taking $\lambda = 1 - 1/{\alpha} \in (0,1)$, we get
\[
\PP\left(\lambda(\mathcal V(x))^{-1}\geq e^{N(\alpha + \log(2))}\right)\leq (\alpha e^{1-\alpha})^N.
\]
Moreover, for any $\alpha \in (0, 1)$ and $\lambda > 0$,
\[
\PP\left(\lambda(\mathcal V(x)) \geq e^{-N(\alpha - \log(2))}\right)\leq \EE\left[\prod_{i=1}^N M_i^{\lambda}\right]e^{N(\alpha - \log(2)) \lambda} = \left(\frac{(2 - 2^{-\lambda}) e^{(\alpha - \log(2)) \lambda}}{1+\lambda}\right)^N
\]
where $M_i = \max(S_i, 1-S_i) \sim \mathcal{U}(1/2,1)$. Using the upper bound $2 - 2^{-\lambda} \leq 1 + \lambda \log(2) \le e^{\lambda \log(2)}$, and choosing $\lambda = {1}/{\alpha} - 1 > 0$ gives
\[
\mathbb{P}\left(\lambda(\mathcal{V}(x)) \geq e^{-N(\alpha - \log(2))}\right) \leq (\alpha e^{1 - \alpha})^N.
\]

\subsection*{Proof of Proposition \ref{cor_diamvolURT}}

We will use the Borel-Cantelli Lemma together with the inequalities obtained in Propositions \ref{prop:diamURT2} and \ref{prop:diam_vol_URT2}. To prove the upper bound on the diameter, we provide values $\theta_N$ leading to small enough probabilities. More precisely, by taking $\theta_N=2\sqrt{\log(N)/(dN)}$, we get $e^{-Nd\theta_N^2/2}=N^{-2}$.   Then  
    \begin{equation*}
    \PP\left(\diam (\mathcal{V}(x))\geq \sqrt{d}e^{N(-(1-\log(2))/d+\theta_N)}\right) \leq \dfrac{d}{N^2}.
\end{equation*} 
The Borel-Cantelli Lemma then gives $$\PP\left(\liminf_{N \to + \infty} \left\{ \diam (\mathcal{V}(x)) \leq \sqrt{d} e^{N(-(1-\log(2))/d + \theta_N)} \right\}\right) = 1.$$
This means that almost surely, beyond a certain rank, we have $$\diam (\mathcal{V}(x)) \leq \sqrt{d} e^{N(-(1-\log(2))/d + \theta_N)}.$$
For the lower bound on the diameter, we proceed in the same way with the choice $\tilde \theta_N=4\sqrt{3\log(N)/(dN)}$, or equivalently $e^{-Nd \tilde \theta_N^2/24}=N^{-2}$. We deduce that, almost surely, beyond a certain rank, $$\diam (\mathcal{V}(x)) \geq \sqrt{d} e^{-N((1+\log(2))/d + \tilde \theta_N)}.$$
Now, regarding the volume, we set $\alpha_N=1+2\sqrt{\log(N)/N}$ and we obtain \[
    \left(\alpha_N e^{1-\alpha_N}\right)^N = \exp\left( {-2\sqrt{N\log(N)}+N\log\left(1+2\sqrt{\log(N)/N}\right)}\right).\]
As $\log\left(1+2\sqrt{\log(N)/N}\right)=2\sqrt{\log(N)/N} - 2\log(N)/N+ O\left((\log(N)/N)^{3/2}\right)$, we get $$\left(\alpha_N e^{1-\alpha_N}\right)^N = \exp \left( -2\log(N) + O\left(\log(N)^{3/2} / \sqrt{N} \right) \right) \sim 1/N^2.$$
The Borel-Cantelli Lemma gives us that almost surely, beyond a certain rank $N_0$, we have
$$\forall N \geq N_0, \qquad  \lambda(\mathcal{V}(x)) \geq e^{-N (1+\log(2)) - 2\sqrt{N\log(N)}}.$$
For the upper bound on the volume, we set for $N \geq 9$, $\tilde \alpha_N=1 - 2\sqrt{\log(N)/N} \in (0,1)$ and we obtain \[
    \left(\tilde \alpha_N e^{1-\tilde \alpha_N}\right)^N = \exp\left( {2\sqrt{N\log(N)}+N\log\left(1-2\sqrt{\log(N)/N}\right)}\right).\]
As $\log\left(1-2\sqrt{\log(N)/N}\right)=-2\sqrt{\log(N)/N} - 2\log(N)/N+ O\left((\log(N)/N)^{3/2}\right)$, we get $$\left(\tilde \alpha_N e^{1-\tilde \alpha_N}\right)^N = \exp \left( -2\log(N) + O\left(\log(N)^{3/2} / \sqrt{N} \right) \right) \sim 1/N^2.$$
Again, the Borel-Cantelli Lemma implies that almost surely, beyond a certain rank $N_0$, we have
$$\forall N \geq N_0, \qquad \lambda(\mathcal{V}(x))  \leq e^{-N(1-\log(2)) + 2\sqrt{N\log(N)}}.$$
Finally, the last inequality stated in Proposition \ref{cor_diamvolURT} comes readily by using the two previous inequalities on the diameter and the volume. \qed

\subsection*{Proof of Theorem \ref{cor_URT}}

We can apply Corollary \ref{sousgauss} pointwise for $x \in S_X$ and, with $\delta = n^{-2}$, we find that, whenever $n P^X (\mathcal{V}(x) ) / \log(n) \geq 16$, it holds that
$$ \sum_{n\geq 1} \mathbb P ( |\hat g_{\mathcal V}(x) - g(x)| > v_n ) < \infty,$$
where $$v_n = \sqrt{ 8 \sigma^2  \log( n  ) / (n  b \lambda ( \mathcal V(x)) )}  + L(\mathcal V(x)) \diam (\mathcal V(x))$$
since $P^X(\mathcal V(x)) \geq b \lambda(\mathcal V(x))$ by \ref{cond:density_XCART}.

Applying the Borel-Cantelli Lemma, we get that with probability $1$, for $n$ large enough,
$$|\hat g_{\mathcal V}(x) - g(x)| \leq  \sqrt{\frac{ 8 \sigma^2  \log( n)}{n  b \lambda ( \mathcal V(x)) }} + L(\mathcal V(x)) \diam (\mathcal V(x)).$$

Let $c = 1- \log(2)$, $C = 1+\log(2)$ and $\Theta = C/c$. From Proposition \ref{cor_diamvolURT} and using that $ N = d\log(n)/(Cd+2c)$, with probability $1$, for a sufficiently large \(n\), we have
\begin{align*}
& \lambda(\mathcal{V}(x)) \geq  e^{-C N - 2\sqrt{N\log(N)}}  \geq n^{-Cd/(Cd+2c)} e^{- 2 \sqrt{\log(n) \log(\log(n)) }  },
\end{align*}
where we have used that \(N \leq \log(n)\). Using \ref{cond:density_XCART}, it follows that
$$ P^X (\mathcal{V}(x) ) \geq b \lambda(\mathcal{V}(x))  \geq b n^{-Cd/(Cd+2c)} e^{- 2 \sqrt{\log(n) \log(\log(n)) }  }.$$
As a consequence,
$$ n P^X (\mathcal{V}(x) ) \geq b n ^{ 2c /(Cd+2c) } e^{- 2 \sqrt{\log(n) \log(\log(n)) }  }.$$
Hence, we get that with probability $1$, $ n P^X (\mathcal{V}(x) ) / \log(n) \to \infty$. This ensures the previous assumption $n P^X (\mathcal{V}(x) ) / \log(n) \geq 16$, in order to apply Corollary \ref{sousgauss}.

Finally, we have the following inequality, with probability $1$, for $n $ large enough and $\delta = n^{-2}$, $$|\hat g_{\mathcal V}(x) - g(x)| \leq \sqrt{\frac{ 8 \sigma^2  \log\left( n \right)}{n  b \lambda ( \mathcal V(x)) }} + L(\mathcal V(x)) \diam (\mathcal V(x)).$$ 
This gives in virtue of Proposition \ref{cor_diamvolURT} 
$$| \hat g_{\mathcal V}(x) - g(x)| \leq \sqrt{\frac{  8 \sigma^2  \log\left( n \right)}{n b e^{-CN - 2\sqrt{N \log(N)}}}} + L(\mathcal V(x)) \sqrt d e^{-cN/d + 2\sqrt{N \log(N)/d}}.$$
Recalling that $ N = d\log(n)/(Cd+2c)$, we obtain
\begin{align*}
    &| \hat g_{\mathcal V}(x) - g(x)| \\ & \leq n^{-1/(\Theta d+2)} e^{\sqrt{N \log(N)}} \sqrt{\frac{  8 \sigma^2  \log\left( n \right)}{b}} + n^{-1/(\Theta d+2)} L(\mathcal V(x)) \sqrt d e^{2\sqrt{N \log(N)/d}}.
\end{align*}
Moreover, from $N = {d \log(n)}/{(Cd+2c)}$, we get $\sqrt{\log(n)} \le \sqrt{C+2c/d} \sqrt{N} \leq \sqrt{3} \sqrt{N}$. Applying the general inequality $\sqrt{x} \le \exp\left(\sqrt{x \log x}\right)$ (with $x = N$), we directly obtain $$\sqrt{\log(n)} \le \sqrt{3} \exp\left(\sqrt{N \log N}\right).$$
Then, since $2/\sqrt{d} \leq 2,$ we have
\[
| \hat g_{\mathcal V}(x) - g(x)| \leq \bar{C} n^{-1/(\Theta d+2)} e^{2\sqrt{N \log(N)}} ,
\]
where $\bar{C}_x = \sqrt{{  24 \sigma^2 }/{b}} + {L(\mathcal V(x)) \sqrt d} \leq \sqrt{{  24 \sigma^2 }/{b}} + {L \sqrt d} = \bar{C}$. Additionally, we have $N \log(N) \leq \log(n) \log(\log(n))$ since $N \leq \log(n)$, and hence the stated inequality. \qed

\subsection*{Proof of Proposition \ref{uniform tree not regular}}

At each stage, for each terminal leaf, draw uniformly $D_i$ in $\{1,\ldots, d\}$ as well as a uniform random variable $S_i$. Then we divide the cell according to coordinate $k = D_i$. The corresponding length $h_k(V)$ is then updated into $h_k(V) S_i$ and $h_k(V) (1-S_i)$. Note that $1-S_i$ is still uniformly distributed. As a consequence, for any given cell $V \in \mathcal{V}$ after $N$ stages, the $k$-th side length has the following representation
$$ h_k(V) = S_1^{B_1^{(k)}}\times \ldots \times S_N^{B_N^{(k)}} = \exp\left( \sum_{i=1}^N B_i^{(k)} \log(S_i) \right), $$
where $B_i^{(k)} = \ind_{D_i = k}$. It follows that
\begin{align*}
 &h_+(V ) = \exp\left( \max_{k=1,\ldots, d}  \sum_{i=1} ^ N B_i^{(k)} \log(S_i) \right), \\
 &h_-(V ) =  \exp\left(\min_{k=1,\ldots, d}  \sum_{i=1} ^ N B_i^{(k)} \log(S_i) \right), 
\end{align*}
and the expression of the ratio is
\begin{align*}
 h_+(V )/h_-(V )  &= \exp \left(   \max_{1 \leq k,j \leq d} \sum_{i=1} ^N (B_i^{(k)} - B_i^{(j)} ) E_i \right) \end{align*}
 where $E_i = - \log(S_i)$ follow an exponential distribution with parameter 1.

By denoting $V_i^{k,j} = B_i^{(k)} - B_i^{(j)}$, we get
\[V_i^{k,j} = \begin{cases} 
1 & \text{with probability } 1/d \\
0 & \text{with probability } 1 - 2/d \\
-1 & \text{with probability } 1/d 
\end{cases}.\]
Note that the variables \((V_i^{k,j})_{i=1}^N\) are mutually independent because the \((D_i)_{i=1}^N\) are independent. Furthermore, since the \(S_i\)'s are independent of the \(V_i\)'s, the \(V_i^{k,j}\)'s are independent of the \(E_i\)'s.
Let \(Z_{k,j} = \sum_{i=1}^{N} V_i^{k,j} E_i\) such that
\[
\frac{h_+(V)}{h_-(V)} = \exp \left( \max_{1 \leq k,j \leq d} Z_{k,j} \right).
\]
Note that \(Z_{k,j} = - Z_{j,k}\) and \(Z_{k,k} = 0\), which gives
\[
\max_{1 \leq k,j \leq d} Z_{k,j} = \max_{1 \leq k < j \leq d} | Z_{k,j} |
\]
and thus the formula 
\[
\frac{h_+(V )}{h_-(V )} = \exp \left( \max_{1 \leq k < j \leq d} | Z_{k,j} | \right).
\]
By using the Paley-Zygmund inequality to $Z_{k,j}^2$, we get for all $\theta \in (0,1)$,
\[
\PP\left(|Z_{k,j}| \geq \sqrt{\theta} \sqrt{\mathbb{E}(Z_{k,j}^2)} \right) \geq (1-\theta)^2 \frac{\mathbb{E}(Z_{k,j}^2)^2}{\mathbb{E}(Z_{k,j}^4)}.
\]
We therefore seek to calculate the 2nd and 4th moments of \(Z_{k,j}\). 
\\ Since \(Z_{k,j}^2 = \sum_{i \neq \ell} V_i^{k,j} V_\ell^{k,j} E_i E_\ell + \sum_{i=1}^N V_i^{k,j \, 2} E_i^2 \), \(\mathbb{E}(V_i^{k,j}) = 0\) and by independence along the subscripts, we obtain 
$$\mathbb E(Z_{k,j}^2) = \sum_{i=1}^N \mathbb{E}((V_i^{k,j})^2) \mathbb{E}(E_i^2) = N \times \frac{2}{d} \times 2 = \frac{4N}{d}.$$
Moreover, according to Lemma \ref{lemme moment} applied to \( M_i := V_i^{k,j} E_i \), we obtain 
\begin{eqnarray*}
    \mathbb{E}(Z_{k,j}^4) &=& N \mathbb{E}(M^4) + 3N (N-1) \mathbb{E}(M^2)^2 \\ &=& N \mathbb{E}((V_i^{k,j})^4) \mathbb{E}(E_i^4) + 3N (N-1) \mathbb{E}((V_i^{k,j})^2)^2 \mathbb{E}(E_i^2)^2.
\end{eqnarray*}
Indeed, it is easily checked that the variables \( (M_i)_{i=1}^N \) are centered and independent, due to the independence between the elements of the collections \( (V_i^{k,j})_{i=1}^N \) and \( (E_i)_{i=1}^N \) and the fact that the $V_i^{k,j}$ are centered. Basic calculations then give
\begin{eqnarray*}
    \mathbb{E}(Z_{k,j}^4) &=& N \mathbb{E}({V_i^{k,j \, 4}}) \mathbb{E}(E_i^4) + 3N (N-1) \mathbb{E}({V_i^{k,j \, 2}})^2 \mathbb{E}(E_i^2)^2 \\ &=& N \times \dfrac{2}{d} \times  4! + 3N(N-1) \left( \dfrac{2}{d}\right)^2 \times 2^2 \\ &=& \dfrac{48N}{d^2} (d+N-1).
\end{eqnarray*}
Consequently, we get
\[
\frac{\mathbb{E}(Z_{k,j}^2)^2}{\mathbb{E}(Z_{k,j}^4)} = \frac{16N^2}{d^2} \times \frac{d^2}{48N (d + N -1)} = \frac{N}{3(d + N -1)}
\]
and thus, for all \(\theta \in (0,1)\),
\[
\PP\left(|Z_{k,j}| \geq \sqrt{\theta} \sqrt{\frac{4N}{d}}\right) \geq (1-\theta)^2  \frac{N}{3(d + N -1)}.
\]
In particular, for \(N \geq d\), we have $3(d + N - 1) \leq 6N$, which gives
\[
\PP\left(|Z_{k,j}| \geq \sqrt{\theta} \sqrt{\frac{4N}{d}}\right) \geq (1-\theta)^2 / 6.
\]
With the choice \(\theta = 1/4\), it holds
\[
\PP\left(|Z_{k,j}| \geq \sqrt{\frac{N}{d}}\right) \geq \frac{9}{16} \times \frac{1}{6} = \frac{3}{32} \geq \frac{1}{11}.
\]
Finally, by the following lower bound,
\[
\frac{h_+(V )}{h_-(V )} = \exp \left( \max_{1\leq k < j \leq d } |Z_{k,j}| \right) \geq \exp (|Z_{1,2}|),
\]
we get, for any $N \geq d$, 
\begin{align*}
\PP\left(\frac{h_+(V )}{h_-(V)} 
\geq \exp\left(\sqrt{\frac{N}{d}}\right)\right) &\geq \PP\left(\exp(|Z_{1,2}|) \geq \exp\left(\sqrt{\frac{N}{d}}\right)\right) \\ 
&= \PP\left(|Z_{1,2}| \geq \sqrt{\frac{N}{d}}\right) \geq \frac{1}{11}.    
\end{align*} \qed

\subsection*{Proof of Proposition \ref{mondrian}}

According to \cite[Proposition 1]{minimaxmondrian}, for each coordinate $j$, the cell $\mathcal V(x)$ is given by $\prod_j [\max(x_j - L_j, 0), \min(x_j + R_j, 1)]$ where $L_j, R_j \sim \mathrm{Exp}(\Lambda)$ i.i.d., so that the side length in direction $j$ is
\[
h_j = \min(L_j, x_j) + \min(R_j, 1-x_j).
\]
Note that $h_j$ is not distributed as a $\Gamma(2,\Lambda)$ random variable because of the truncation at $x_j$ and $1-x_j$. However, letting $\xi_j := L_j + R_j \sim \Gamma(2,\Lambda)$ denote the untruncated version, we have $h_j \leq \xi_j$ almost surely, and the two tail bounds needed below still hold.

For the lower tail, let $u \leq 1$ and choose $a = \min(u,x_j) \leq x_j$ and $b = u-a \leq 1-x_j$ (possible since $u \leq 1 = x_j + (1-x_j)$). On the event $\{L_j \geq a\} \cap \{R_j \geq b\}$ we have $\min(L_j,x_j) \geq a$ and $\min(R_j,1-x_j)\geq b$, hence $h_j \geq a+b = u$. By independence of $L_j$ and $R_j$,
\[
\mathbb{P}(h_j \geq u) \geq \mathbb{P}(L_j \geq a)\,\mathbb{P}(R_j \geq b) = e^{-\Lambda a}e^{-\Lambda b} = e^{-\Lambda u}.
\]
So $\mathbb{P}(h_j \geq u) \geq e^{-\Lambda u}$ for all $u \in [0,1]$.

For the upper tail, since $h_j = \min(L_j,x_j)+\min(R_j,1-x_j) \leq L_j+R_j = \xi_j$, we have $\mathbb{P}(h_j \geq t) \leq \mathbb{P}(\xi_j \geq t)$ for all $t \geq 0$. It therefore suffices to bound the upper tail of the untruncated variable $\xi_j \sim \Gamma(2,\Lambda)$.

Since the $(L_j,R_j)_{j=1,\dots,d}$ are independent across coordinates, for $u \leq 1$,
\[
\mathbb{P}(h_-(\mathcal V(x)) \geq u) = \prod_{j=1}^d \mathbb{P}(h_j \geq u) \geq \prod_{j=1}^d e^{-\Lambda u} = e^{-\Lambda u d} = 1-\delta,
\]
for $u = -{\log(1-\delta)}/{(\Lambda d)}$ (which satisfies $u \leq 1$ as soon as $\delta \leq 1-e^{-\Lambda d}$). Then, when $\delta \leq 1-e^{-\Lambda d}$, we have with probability at least $1-\delta$,
\begin{equation}\label{h-}
h_-(\mathcal V(x)) \geq -\frac{\log(1-\delta)}{\Lambda d}.
\end{equation}
We focus now on $h_+(\mathcal V(x))$. Since $h_j \leq \xi_j$ with $\xi_j \sim \Gamma(2,\Lambda)$ i.i.d., we have for all $t \geq 0$
\[
\mathbb{P}(h_+(\mathcal V(x)) \leq t) = \prod_{j=1}^d \mathbb{P}(h_j \leq t) \geq \prod_{j=1}^d \mathbb{P}(\xi_j \leq t) = \mathbb{P}(\xi \leq t)^d,
\]
where $\xi \sim \Gamma(2,\Lambda)$. It thus suffices to bound the upper tail of $\max_j \xi_j$. Let $Y := \xi - \mathbb{E}(\xi)$. Since $\xi$ follows a Gamma distribution, $Y$ is sub-Gamma. According to \cite[p.29]{boucheron2013concentration},
\[
\forall t > 0, \quad \mathbb{P}(\Lambda Y \geq 2\sqrt{t} + t) \leq e^{-t}.
\]
Thus,
\begin{eqnarray*}
    \mathbb{P}(\Lambda h_+(\mathcal{V}(x)) \leq 2 \sqrt{t} + t + \Lambda \mathbb{E}(\xi)) &\geq& \mathbb{P}(\Lambda Y \leq 2 \sqrt{t} + t)^d = \left(1 - \mathbb{P}(\Lambda Y > 2 \sqrt{t} + t)\right)^d \\ &\geq& (1 - e^{-t})^d = 1 - \delta,
\end{eqnarray*}
with $t = -\log(1-(1-\delta)^{1/d})$. Then, with probability at least $1-\delta$,
\[
h_+(\mathcal V(x)) \leq \frac{2 + 2\sqrt{-\log(1-(1-\delta)^{1/d})} - \log(1-(1-\delta)^{1/d})}{\Lambda}.
\]
In particular, for $\delta \leq 1 - (1-e^{-1})^d$,
\begin{equation}\label{h+}
h_+(\mathcal V(x)) \leq \frac{-5 \log(1-(1-\delta)^{1/d})}{\Lambda} \leq \frac{-5\log(\delta/d)}{\Lambda},
\end{equation}
where the last inequality comes from the inequality $\delta/d \leq 1-(1-\delta)^{1/d}$.

Hence, with \eqref{h-} and \eqref{h+}, we have, for all $\delta \leq \min\left(1 - e^{-\Lambda d}, \; 1 - (1 - e^{-1})^d\right)$, the following inequality with probability at least $1-2\delta$, 
\[
\frac{h_+(\mathcal V(x))}{h_-(\mathcal V(x))} \leq \frac{5d\log(\delta/d)}{\log(1-\delta)}.
\]
\qed

\subsection*{Proof of Theorem \ref{thm:mondrian}}
Let \(x \in S_X\).  Set $\varepsilon_i := Y_i - g(X_i)$ for each $i=1,\ldots, n$. We write the bias-variance decomposition $\hat g_{\mathcal V}(x) - g(x) = W + B$, where \begin{align*}
    W := \dfrac{\sum_{i=1}^n \varepsilon_i \ind_{ \mathcal V(x) }(X_i)}{\sum_{j=1}^n \ind_{ \mathcal V(x) }(X_j)} \qquad \text{and}\qquad
       B := \dfrac{\sum_{i=1}^n \left(g(X_i) - g(x)\right) \ind_{ \mathcal V(x) }(X_i)}{\sum_{j=1}^n \ind_{ \mathcal V(x) }(X_j)}.
\end{align*}
Let us recall Inequality \eqref{h-} obtained in the previous proof, with probability at least \(1 - \delta\),  $$h_-(\mathcal V (x)) \geq - \frac{\log(1-\delta)}{\Lambda d}$$ for all $\delta \leq 1-e^{-\Lambda d}$. We thus have, whenever $n b\geq 8  c_{\delta,d} \Lambda^d$, that the inequality 
$$n P^X(\mathcal{V}(x)) \geq n b h_-^d \geq \frac{nb \log(1/(1-\delta))^d} { (\Lambda d)^d} = \frac{nb \log(1/\delta) }{\Lambda^d c_{\delta,d} } \geq 8 \log(1/\delta)$$
holds with probability at least $1-\delta$. Let $E_1$ be the event from previous equation. Let $E_2 $ be the event such that
$$\left| \dfrac{\sum_{i=1}^n \varepsilon_i \ind _{ \mathcal V(x) }(X_i)}{{\sum_{j=1}^n \ind _{ \mathcal V(x) }(X_j)}}  \right| \leq \sqrt{\dfrac{4 \sigma^2 \log(1/\delta)}{n P^X(\mathcal{V}(x))} } .$$
On $E_1\cap E_2 $, it holds
\[
\left| \dfrac{\sum_{i=1}^n \varepsilon_i \ind _{ \mathcal V(x) }(X_i)}{{\sum_{j=1}^n \ind _{ \mathcal V(x) }(X_j)}}  \right| \leq \sqrt{\dfrac{4\sigma^2 \log(1/\delta)}{n P^X(\mathcal{V}(x))}} \leq \sqrt{\dfrac{4\sigma^2 c_{\delta,d} \Lambda^d  }{n b}}\;.
\]
It remains to check that $\mathbb P ( E_1\cap E_2) \geq 1-4\delta$. Note that 
$A\cup B = A \cup  (A^c\cap B)  $
which, when applied to $ A = E_1^c$ and $B = E_2^c $, gives $\mathbb P ( E_1^c\cup E_2^c) = \mathbb P (  E_1^c) + \mathbb P ( E_1\cap E_2^c)   $. The first term $ \mathbb P (  E_1^c)$ is smaller than $\delta$, as shown before.
 According to Corollary \ref{sousgauss} (see equation \eqref{eq10}), we have $\mathbb P ( E_1\cap E_2^c |\mathcal{V}(x)) $ is smaller than $3\delta$. Integrating with respect to $\mathcal{V}(x)$, we obtain $\mathbb P ( E_1\cap E_2^c )\leq 3\delta $. 
As for the bias term, for $\delta \leq 1 - (1 - e^{-1})^d$, it was shown in \eqref{h+} (see previous proof) that, with probability at least $1-\delta$,  \[ h_+(\mathcal V (x)) \leq \dfrac{-5\log(\delta/d)}{\Lambda}\;.\]
Observing that $ 1 - (1 - e^{-1})^d \geq 1/5$, it follows that, with probability at least $1-\delta$, 
$$ |B|\leq L (\mathcal V(x)) \diam(\mathcal{V}(x)) \leq L (\mathcal V(x))  \sqrt{d} \, h_+(\mathcal{V}(x)) \leq L (\mathcal V(x))   \sqrt{d}\, 5 \, \frac{\log(d/\delta)}{\Lambda}\,.$$ 
Thus, putting together the obtained bounds on $|W|$ and $|B|$, we find, for all $\delta \leq 1-e^{-\Lambda d}$, with probability at least $1-5\delta$, 
  \[
 | \hat g_{\mathcal V}(x) - g(x)| \leq \sqrt{\dfrac{4\sigma^2 c_{\delta,d} \Lambda^d }{n b }} +  5 \sqrt{d} \, L(\mathcal{V}(x)) \frac{\log(d/\delta)}{\Lambda}\;.
\]

In addition, choosing $\Lambda \asymp n^{1/(d+2)}$ yields
  \[
 | \hat g_{\mathcal V}(x) - g(x)| \lesssim \dfrac{C}{n^{1/(d+2)}}\;,
\]
with $C = \sqrt{{4\sigma^2 c_{\delta, d}   }/ {b }} +  5 \sqrt{d} \, L(\mathcal{V}(x))\log(d/\delta) \leq \sqrt{{4\sigma^2 c_{\delta, d}   }/ {b }} +  5 \sqrt{d} \, L\log(d/\delta)$. Note that for $\delta \leq 1/5$, the condition $\delta \leq 1-e^{-\Lambda d}$ is satisfied whenever $\Lambda \geq 1/4$. This explicitly holds in our setting since $\Lambda \asymp n^{1/(d+2)} \geq 1$. \qed

\subsection*{Proof of Theorem \ref{th:main_th_proto}}

    Let \(x \in S_X\). Set $\varepsilon_i := Y_i - g(X_i)$ for each $i=1,\ldots, n$. We write the bias-variance decomposition 
$\hat g_{\mathcal V}(x) - g(x) = W + B$, where \begin{align*}
    W := \dfrac{\sum_{i=1}^n \varepsilon_i \ind_{ \mathcal V(x) }(X_i)}{\sum_{j=1}^n \ind_{ \mathcal V(x) }(X_j)} \qquad \text{and}\qquad
       B := \dfrac{\sum_{i=1}^n \left(g(X_i) - g(x)\right) \ind_{ \mathcal V(x) }(X_i)}{\sum_{j=1}^n \ind_{ \mathcal V(x) }(X_j)}.
\end{align*}
Each of the above terms will be treated in two independent propositions.

\begin{proposition}\label{prop_lemme}
    Let $x\in S_X$ and assume that \ref{cond:D}, \ref{cond:D1}, \ref{cond:reg2} and \ref{cond:reg3} are fulfilled.
Let $\delta \in (0,1/4)$. If $n \geq 1,m \geq 3$ and $$ \dfrac{n}{m} \geq \dfrac{ 8 \psi_d \log(1/\delta)}{ \delta^d},$$
with $\psi_d = (2^d 18 M d )^d (2c_d b)^{-d}$,
    then with probability at least $1- 4\delta$,
    $$\left| \dfrac{\sum_{i=1}^n \varepsilon_i \ind _{ \mathcal V(x) }(X_i)}{{\sum_{j=1}^n \ind _{ \mathcal V(x) }(X_j)}} \right| \leq\sqrt{\dfrac{ 4 \psi_d  \sigma^2   \log(1/\delta)}{\delta^d }} \sqrt{\dfrac{m}{n}}.$$   
\end{proposition}

\begin{proof}
The proof is in two steps. As a first step we show that with probability at least $1-\delta$,
\begin{align*}
n  P^X (\mathcal{V}(x) ) \geq \dfrac{n}{m} \psi_d^{-1} \, \delta^d   ,
\end{align*}
with \( \psi_d \) defined in the statement.
As a second step, we rely on Corollary \ref{sousgauss} to obtain the stated upper bound. 

\noindent \textit{Step 1:} Invoking \ref{cond:reg2}, we apply Lemma \ref{Z vers X}, (a) and (b), to $Z \sim X$ and $W := \|Z - x\|$
 to obtain that $f_W(\rho) \leq c_2 \rho^{d-1} $ whenever $\rho\leq \rho_0$ and $f_W (\rho ) \leq c_2  2^{d-1} \rho_0^ {d-1} $ for all $\rho \geq 0$, with $c_2 = Md V_d$. Similarly, we apply Lemma \ref{Z vers X} (c) in light of assumption \ref{cond:reg2} to get that $ F_W(\rho) \geq c_1 \rho ^d$, for all $\rho \leq \rho_0 = T_0$, with $c_1 = c_d b V_d$. This allows to  apply Lemma \ref{condition F vers kappa} with \( c_1 = c_d b V_d\), \(c_2 = Md V_d \), \(U = c_2  2^{d-1}\rho_0^{d-1} \), to obtain the inequality, for all $t\geq 0$,
\[
f_W(t) \leq c_0 F_W(t)^\kappa
\]  
where \( \kappa = 1 - 1/d \) and $ c_0 =  c_1^{-\kappa} \max(c_2,  c_2 {(2\rho_0/T_0)^{d-1}}) =  2^{d-1} c_1^{-\kappa} c_2 $. For all $i=1,\ldots, m$, let $ W_i =\| Z_i - x\| $ and $W_{(i)} $ the ordered statistics $ W_{(1)} \leq W_{(2)} \leq  \ldots \leq   W_{(m)}.$ 
We are now in position to apply Lemma \ref{kappa} with \( \kappa = 1 - 1/d \) and  $ c_0 =  2^{d-1} c_1^{-\kappa} c_2 $ as defined above, to obtain that, with probability at least \(1 - \delta\), 
$$W_{(2)} -W_{(1)}  \geq    C^{-1} \delta  m^{ -1/d } \geq \overline{C}^{-1} \delta  m^{ -1/d } ,$$ where $C = c_0 \Gamma(2-1/d) 3^{2-1/d}$ and   $\overline{C}= 9 c_0 = 9 \times  2^{d-1} Md V_d^{1/d} {(c_d b)^{-1+1/d}}  $,  satisfy $\overline{C} \geq C$ since \( \Gamma(2-1/d)  \leq 1 \) and $3^{2-1/d}\leq 9 $. Moreover, since \( f_X \) has compact support included in $B (0, \rho_0)$, we have for all $(i,j) \in \{1, \dots, m\}^2$ 
\[  W_{(2)} -  W_{(1)}  \leq \| Z_i - Z_j \| \leq \|Z_i\| + \| Z_j\| \leq 2 \rho_0 = 2T_0.\]  
Recall that we have shown that $ c_1 \rho ^d \leq  F_W(\rho) = P^Z(B(x,\rho)) = P^X(B(x,\rho))$, for all $\rho \leq \rho_0 = T_0$. Thus, we can apply Lemma \ref{PV} with $c_3 = c_1$, \( T_1 = T_0 \) and with the distribution $P = P^X$, which yields, with probability at least $1-\delta$,
\[
 P^X(\mathcal{V}(x) )\geq \frac{c_1}{2^d}   (W_{(2)} -  W_{(1)}) ^d \geq \frac{c_1}{2^d \overline{C}^{d}}  \delta^d  m^{ -1 }= \psi_d^{-1} \delta^d  m^{ -1 },
\]  
where 
\( \psi_d =  c_1^{-1} (2\overline{C})^{d} \). 
Note in particular that, as soon as \( n\geq 8\psi_dm\log(1/\delta)/ \delta^d \), we get the inequality  
\[
n P^X(\mathcal{V}(x) ) \geq \dfrac{n}{m} \psi_d^{-1} \, \delta^d \geq 8\log(1/\delta),
\]  
that is valid with probability at least $1-\delta$. 

\noindent \textit{Step 2:} Let $E_1$ be the event from previous equation.  Let $E_2 $ be the event such that
$$\left| \dfrac{\sum_{i=1}^n \varepsilon_i \ind _{ \mathcal V(x) }(X_i)}{{\sum_{j=1}^n \ind _{ \mathcal V(x) }(X_j)}}  \right| \leq \sqrt{\dfrac{4 \sigma^2 \log(1/\delta)}{n P^X(\mathcal{V}(x))} } .$$
It is easy to see that $E_1$ and $E_2 $ imply that 
\[
\left| \dfrac{\sum_{i=1}^n \varepsilon_i \ind _{ \mathcal V(x) }(X_i)}{{\sum_{j=1}^n \ind _{ \mathcal V(x) }(X_j)}}  \right| \leq \sqrt{\dfrac{4\sigma^2 \log(1/\delta)}{n P^X(\mathcal{V}(x))}} \leq \sqrt{\dfrac{4\psi_d \sigma^2 \log(1/\delta)}{\delta^d}} \sqrt{\dfrac{m}{n}}.
\]
It remains to check that $\mathbb P ( E_1\cap E_2) \geq 1-4\delta$. Note that 
$A\cup B = A \cup  (A^c\cap B)  $
which, when applied to $ A = E_1^c$ and $B = E_2^c $, gives $\mathbb P ( E_1^c\cup E_2^c) = \mathbb P (  E_1^c) + \mathbb P ( E_1\cap E_2^c)   $. The first term $ \mathbb P (  E_1^c)$ is smaller than $\delta$ as shown before in Step 1.
 According to assumption \ref{cond:reg3} and Corollary \ref{sousgauss} (equation \eqref{eq10}), we have $\mathbb P ( E_1\cap E_2^c | Z_{1}, \dots, Z_m) $ is smaller than $3\delta$. Integrating with respect to $Z_{1} , \ldots, Z_{m}$, we obtain $\mathbb P ( E_1\cap E_2^c )\leq 3\delta $. 
\end{proof}

\begin{proposition}
Suppose that \ref{cond:D}, \ref{cond:D1}, \ref{cond:reg2} and \ref{cond:reg4} hold true. Then, for all $m\geq 1$, $\delta \in (0,1)$ such that $ 32 d \log(12m / \delta )  \leq T_0^d  m b c_d V_d $, it holds, with probability at least $1-\delta$,
$$ \left|  \dfrac{\sum_{i=1}^n \left(g(X_i) - g(x)\right) \ind_{ \mathcal V(x) }(X_i)}{\sum_{j=1}^n \ind_{ \mathcal V(x) }(X_j)} \right| \leq  2 L(\mathcal{V}(x))  \left(\frac{ 32d \log(12m/\delta) }{ mb c_d V_d }  \right)^{1/d} .$$
\end{proposition}

\begin{proof}
Using triangle inequality we obtain
    \begin{align*}
     &\left|  \dfrac{\sum_{i=1}^n \left(g(X_i) - g(x)\right) \ind_{ \mathcal V(x) }(X_i)}{\sum_{j=1}^n \ind_{ \mathcal V(x) }(X_j)} \right|\\ 
     &\leq \dfrac{\sum_{i=1}^n \left|g(X_i) - g(x)\right| \ind_{ \mathcal V(x) }(X_i)}{\sum_{j=1}^n \ind_{ \mathcal V(x) }(X_j)} \\ 
     &\leq \dfrac{\sum_{i=1}^n \sup_{y \in \mathcal V(x)} |g(y) - g(x)| \ind_{ \mathcal V(x) }(X_i)}{\sum_{j=1}^n \ind_{ \mathcal V(x) }(X_j)} = \sup_{y \in \mathcal V(x)} |g(y) - g(x)|.
\end{align*}
Moreover, using the Lipschitz assumption, it follows that
$$|g(y) - g(x)| \leq  L(\mathcal V(x)) \|x-y\|_2  \leq L(\mathcal V(x)) \diam  (\mathcal V(x) ).$$
Thus, we obtain
\begin{equation}\label{d1}
     \left|  \dfrac{\sum_{i=1}^n \left(g(X_i) - g(x)\right) \ind_{ \mathcal V(x) }(X_i)}{\sum_{j=1}^n \ind_{ \mathcal V(x) }(X_j)} \right| \leq L(\mathcal{V}(x)) \diam (\mathcal V (x) ).
\end{equation}
Suppose that $x$ and $y$ belong to the Voronoi cell of $Z_i$. That is $ i = \argmin _{j=1,\ldots, m} \| x- Z_j\| = \argmin_{j=1,\ldots, m} \| y- Z_j\| $. Hence 
$$\|x- y\| \leq  \| x- Z_i \| + \| y-Z_i\|=
\min_{j=1,\ldots, m}  \| x- Z_j \| + \min_{j=1,\ldots, m} \| y-Z_j\|.$$
Therefore 
$$  \diam (\mathcal V (x) ) \leq 2 \sup_{u\in S_Z} \hat \tau_1 (u) \leq 2 \sup_{u\in S_Z} \hat \tau_{k} (u) $$
with $k=16d \log(12m/\delta)$ and $\hat \tau_{k} (u) $ is $k$-NN radius 
\begin{align*}
\hat \tau_{k}(u) =  \inf   \{ \tau\geq 0 \, :\,  \sum_{i=1}^m \ind _{ B(u,\tau) }(Z_i) \geq k \} .
\end{align*}
Using Lemma 3 in \cite{portier2021nearest} whenever $2k \leq  T_0^d  m b c_d V_d$, we obtain that, with probability at least $1-\delta$,
\begin{equation}\label{d2}
   \diam (\mathcal V (x) ) \leq  2 \left(\frac{ 32d \log(12m/\delta) }{ mb c_d V_d }  \right)^{1/d}. 
\end{equation}
Finally, the combination of equations \eqref{d1} and \eqref{d2} yields the desired result.   
\end{proof}

Getting back to the proof of Theorem \ref{th:main_th_proto}, the upper bounds on \( W \) and \( B \) from the previous propositions yield the stated result. \qed

\subsection*{Proof of Corollary \ref{cor:main_th_proto}}

It suffices to write the inequality \(\log(12m/\delta) \leq \log(12n/\delta)\) and then observing that the identity \(\sqrt{\log(1/\delta)/\delta^d} \sqrt{m/n} = (\log(n/\delta)/m)^{1/d}\) gives the correct order for $m$. Finally, using this choice of $m$ into the bound yields the stated result. \qed

\subsection*{Proof of Theorem \ref{th:main_th_OptiNet}}

We start by establishing two facts that are related to the $\eta$-net construction. They will be useful to deal with the bias term (Fact 1) and the variance term (Fact 2). For \( A \subset \mathbb{R}^d \) and \( \eta > 0 \), an \( \eta \)-net of \( A \) is any subset \( B \subset A \) such that the distance between any two distinct points in \( B \) is strictly greater than \( \eta \), i.e., \( \forall x, y \in B, x \neq y \Rightarrow \|x - y\| > \eta \),  and such that \( B \) is maximal with respect to this property (i.e., no point from \( A \) can be added to \( B \) without violating the previous condition). 
Let $Z(\eta) = \{ Z_i(\eta), i \in [\![1,m(\eta)]\!] \}$ be an $\eta$-net of the set $(Z_i)_{i=1,\dots,m}$. 

\noindent \textbf{Fact 1.} 
We have $\forall i \in [\![1,m]\!], \, d(Z_i, Z(\eta)) \leq \eta$. Indeed, if there exists $i \in [\![1,m]\! ]$ such that $Z_i \in Z(\eta)$ then $d(Z_i, Z(\eta)) = 0 \leq \eta$. Otherwise if $Z_i \notin Z(\eta)$ and $d(Z_i, Z(\eta)) > \eta$ this denies the fact that $Z(\eta)$ is maximal and therefore contradicts the $\eta$-net construction.

\noindent \textbf{Fact 2.} We have that $ B ( Z_k (\eta), \eta /2)  \subset V_k(\eta)$, where $V_k(\eta)$ is the Voronoi cell containing $Z_k(\eta)$ and relative to the $\eta$-net $Z(\eta)$. The result indeed simply follows from noting that
 $ B ( Z_k (\eta) , \Delta_k (\eta) /2) \subset  V_k (\eta)      $ where $\Delta _k(\eta)  = \min _{i\neq k} \|Z_i(\eta) - Z_k(\eta) \| $ is larger than $\eta $ by construction. 

Let \(x \in S_X\).  Set $\varepsilon_i := Y_i - g(X_i)$ for each $i=1,\ldots, n$. We write the bias-variance decomposition 
$\hat g_{\mathcal V_\eta}(x) - g(x) = W + B$, where \begin{align*}
    W := \dfrac{\sum_{i=1}^n \varepsilon_i \ind_{ \mathcal V_\eta(x) }(X_i)}{\sum_{j=1}^n \ind_{ \mathcal V_\eta(x) }(X_j)} \qquad \text{and}\qquad
       B := \dfrac{\sum_{i=1}^n \left(g(X_i) - g(x)\right) \ind_{ \mathcal V_\eta(x) }(X_i)}{\sum_{j=1}^n \ind_{ \mathcal V_\eta(x) }(X_j)}.
\end{align*}
   We start by considering the bias term $B$. Using triangle inequality we obtain
    \begin{align*}
     \left| B\right|
     &\leq \dfrac{\sum_{i=1}^n \left|g(X_i) - g(x)\right| \ind_{ \mathcal V_\eta(x) }(X_i)}{\sum_{j=1}^n \ind_{ \mathcal V_\eta(x) }(X_j)} \\ 
     &\leq \dfrac{\sum_{i=1}^n \sup_{y \in \mathcal V_\eta(x)} |g(y) - g(x)| \ind_{ \mathcal V_\eta(x) }(X_i)}{\sum_{j=1}^n \ind_{ \mathcal V_\eta(x) }(X_j)} = \sup_{y \in \mathcal V_\eta(x)} |g(y) - g(x)|.
\end{align*}
Moreover, using the Lipschitz assumption, it follows that
$$|g(y) - g(x)| \leq  L(\mathcal V_\eta(x)) \|x-y\|_2  \leq L(\mathcal V_\eta(x)) \diam  (\mathcal V_\eta(x) ).$$
Thus, we obtain
\begin{equation*}
    |B| \leq L(\mathcal{V}_\eta(x)) \diam (\mathcal V _\eta (x) ),
\end{equation*}
we can now provide an upper bound on the diameter of $\mathcal V_\eta (x) $. Let $Z_\eta (x) $ (resp. $Z(x)$) denote the closest point to $x$ among $Z(\eta) $ (resp. $Z_1,\ldots, Z_m$). We have
$$(x,z) \in \mathcal V_\eta (x) ^2 \implies Z_\eta(x) = Z_\eta(z) $$ then $$d(x,z) \leq d(x, Z_\eta(x) ) + d(Z_\eta(x) ,z) = d(x, Z_\eta(x) ) + d(Z_\eta(z) ,z). $$
For the first term we write $d(x, Z_\eta(x) ) = d(x, Z(\eta) )$ and by triangle inequality
$$d(x, Z(\eta) ) \leq d(x, Z(x)) + d(Z(x) , Z(\eta) ) \leq \sup_{x \in S_X  } d(x, Z(x) ) + \eta $$
using \textbf{Fact 1}. It follows that the diameter is such that 
$$ \diam (\mathcal V _\eta (x) ) \leq 2\sup_{x \in S_X  } d(x, Z(x) ) + 2\eta. $$
The first above term is bounded by $2(32d\log(12m/\delta) / [mb c_d V_d])^{1/d}$ with probability at least $1-\delta$, from Lemma 3 in \cite{portier2021nearest} whenever $32d\log(12m/\delta) \leq  T_0^d  m b c_d V_d$. As a consequence, we have shown that, with probability at least $1-\delta$,
$$ |B|\leq 2L(\mathcal V_\eta(x))  \left( \left( \frac{32d\log(12m/\delta)}{  mb c_d V_d}\right)^{1/d}+ \eta \right).$$

Let us now deal with the variance term $W$. As soon as $nP^X(\mathcal{V}_\eta(x)) \geq 8 \log(1/\delta)$ we can apply Corollary \ref{sousgauss} (equation \eqref{eq10}) to obtain the following inequality which holds with probability at least $1 - 3\delta$
\begin{equation}\label{var_eq}
    |W| = \left| \dfrac{\sum_{i=1}^n \varepsilon_i \ind _{\mathcal{V}_\eta(x) }(X_i)}{{\sum_{j=1}^n \ind _{\mathcal{V}_\eta(x)}(X_j)}} \right|     \leq \sqrt{\dfrac{4\sigma^2 \log(1/\delta)}{n P^X(\mathcal{V}_\eta(x))} }.
\end{equation} 
   We can therefore conclude by obtaining a lower bound on $ P^X (\mathcal V_\eta (x) ) $. Let $V_k(\eta)$ denote the Voronoi collection of $Z(\eta)$.  We have  using \textbf{Fact 2},
\begin{align*}
    P^X (\mathcal V_\eta (x) )  
     = \sum _{k=1} ^ {m(\eta)} \ind_{ V_k(\eta) }(x) \, P^X( \mathcal{V}_\eta(x) )  &= \sum _{k=1} ^ {m(\eta)} \ind_{V_k(\eta) }(x) \, P^X(  V_k (\eta) ) \\ &\geq \sum _{k=1} ^ {m(\eta)} \ind_{V_k(\eta) }(x) \, P^X( B ( Z_k(\eta) , \eta /2) ).
\end{align*}
 Moreover if $\eta\leq 2T_0$, using \ref{cond:reg2} to obtain that for all $z \in S_Z$, $$P^X(B ( z , \eta /2)) \geq b \lambda(S_Z \cap B ( z , \eta /2)) \geq b c_d  \lambda(B ( z , \eta /2)) = b c_d V_d \eta^d/2^d,$$ and therefore $P^X (\mathcal{V}_\eta(x) )  \geq b {c_d} V_d \eta^d / {2^d}.$ We then find that, using \eqref{var_eq}, the following inequality holds with probability at least $1-3\delta$, 
 $$|W| \leq \sqrt{\dfrac{2^{d+2} \sigma^2 \log(1/\delta)}{n b c_d V_d \eta^d} }.$$
 Combining the obtained bound on $|B|$ and $|W|$ yields the stated result. \qed

\section{Auxiliary results and technical lemmas}

Let us start with a result establishing some conditional independence property under \ref{cond:D}.

\begin{lemma}\label{lemmaA}
    Assume \ref{cond:D}. Then $(Y_i) _{i = 1,\ldots, n} $ is an independent collection of random variables, conditionally on $(X_i)_{i = 1,\ldots, n}$. 
\end{lemma}

\begin{proof}
    Let $(\phi_i)_{i=1, \dots,n}$ and $(\psi_i)_{i=1, \dots,n}$ be bounded and measurable functions. Then
    \begin{eqnarray*}
    \EE\left( \prod_{i=1}^n \EE\left(\psi_i(Y_i) \middle| X_i \right) \prod_{i=1}^n \phi_i(X_i)\right) &=& \EE\left(\prod_{i=1}^n \EE\left( \psi_i(Y_i) \phi_i (X_i)  | X_i\right)  \right) \\ &=& \prod_{i=1}^n \EE\left( \EE\left( \psi_i(Y_i) \phi_i (X_i) \middle| X_i \right) \right)
    \end{eqnarray*}
because $\EE\left( \psi_i(Y_i) \phi_i (X_i) \middle| X_i \right)$ is $X_i$-measurable and $(X_i)_{i=1,\dots,n}$ are independent. Hence,
    \begin{eqnarray*}
    \EE\left( \prod_{i=1}^n \EE\left(\psi_i(Y_i) \middle| X_i \right) \prod_{i=1}^n \phi_i(X_i)\right) &=& \prod_{i=1}^n \EE\left(\psi_i(Y_i) \phi_i(X_i)\right) \\ &=& \EE\left(\prod_{i=1}^n \psi_i(Y_i) \phi_i(X_i) \right)
    \end{eqnarray*}
by independence of $(X_i, Y_i)_{i=1,\dots,n}$. By definition of conditional expectation, we obtain 
    $$\EE\left(\prod_{i=1}^n \psi_i(Y_i) \middle| (X_j)_{j=1,\dots,n} \right) = \prod_{i=1}^n \EE\left(\psi_i(Y_i) \middle| X_i \right)$$
which means that $(Y_i)_{i=1,\dots,n}$ is an independent collection of random variables conditionally on $(X_i)_{i=1,\dots,n}.$
\end{proof}

The following lemmas will be useful to prove Proposition \ref{prop_lemme}.

\begin{lemma}\label{kappa}
     Let $m \geq 3$. Let $W,W_1,\ldots, W_m$ be independent and identically distributed random variables on $\mathbb R_{+}$ with density $f_W$ and cumulative distribution function $F_W$. Suppose that there exists $c_0>0$ and $\kappa \geq 0$, such that  $ f_W (t) \leq c_0 F_W(t) ^{\kappa} $ for all $t\in  \mathbb R_{+} $. 
     Let $\delta \in (0,1) $. It holds, with probability at least $1-\delta $, 
   $$ W_{(2)} -W_{(1)} >    C^{-1} \delta  m^{ \kappa-1 }  ,$$ 
   where $C=  c_0 \Gamma(\kappa+1) 3^{\kappa+1}.$
\end{lemma}

\begin{proof}
    Note that 
\begin{align*}
 \PP (W_{(2)} -W_{(1)}   > t ) &= \sum_{k=1} ^ m \mathbb P (  W_j > W_k + t \, : \, \forall  j \neq k) \\
 &= \sum_{k=1} ^ m \mathbb E [  (1- F_W ( W_k + t ))^{m-1} ]\\
 &= m \mathbb E [   (1- F_W ( W + t ))^{m-1} ].
  \end{align*}
Let us define, 
for any $t\in \mathbb{R}_+$,
\begin{equation*}
    D(t):= \frac{1}{m}\mathbb{P}(W_{(2)}-W_{(1)}>t) =  \mathbb E [   (1- F_W ( W + t ))^{m-1}  ].
\end{equation*}
It holds $D(0)=1/m$, $D(+\infty)=0$ and by its definition through the deviation probability, $D$ is a non-increasing function on $\mathbb{R}_+$. By using Fubini-Tonelli and integrating first with respect to $t$, we find
\begin{align*}
    &(m-1)\int_0^t \mathbb{E}[f_W(W+u)(1- F_W ( W + u ))^{m-2}] \, du \\ 
    &= \mathbb E\left[ (m-1)\int_0^t f_W(W+u)(1- F_W ( W + u ))^{m-2} \, du \right] \\ &= \mathbb E [   (1- F_W ( W ))^{m-1}  ] - \mathbb E [   (1- F_W ( W + t ))^{m-1}  ] \\ &= D(0)-D(t).
\end{align*}
We also have
\begin{align*}
      & \mathbb{E}[f_W(W+u)(1- F_W ( W + u ))^{m-2}] \\ &=   \int_0^{+\infty} f_W(r+u)f_W(r)(1- F_W ( r + u ))^{m-2} \, \ind_{f_W(r+u) > 0} \, dr \\
    & \leq  c_0 \int_0^{+\infty} f_W(r+u)  F_W(r)^\kappa (1- F_W ( r + u ))^{m-2} \, \ind_{f_W(r+u) > 0} \, dr.
\end{align*}
Let us now apply a change of variable, which is justified because it is differentiable and bijective when $f_W$ is positive. Note also that $ F _W ( F_W^{-1} (u)) = u$ because $F_W$ is continuous. By setting $v=F_W(r+u)$, we get $dv=f_W(r+u)dr$, which gives, for any $\kappa \geq 0$,
\begin{align*}
    &\int_0^{+\infty} f_W(r+u)  F_W(r)^\kappa (1- F_W ( r + u ))^{m-2}\, \ind_{f_W(r+t) > 0} \, dr \\ &= \int_{F_W(u)}^{1}  F_W(F_W^{-1}(v)-u)^\kappa (1- v)^{m-2} \, \ind_{f_W(F^{-1}_W(v)) > 0} \, dv \\ &\leq  \int_{0}^{1} v^\kappa (1- v)^{m-2}dv 
\end{align*}
because $\ind_{f_W(F^{-1}_W(v)) > 0} \leq 1$ and $ F_W(F_W^{-1}(v)-u) \leq F_W(F_W^{-1}(v)) = v.$
Moreover,
\begin{align*}
\int_{0}^{1} v^\kappa (1- v)^{m-2}dv &\leq \int_{0}^{1} v^\kappa \exp(-v(m-2)) dv = \dfrac{1}{(m-2)^{\kappa+1}}\int_{0}^{m-2} s^\kappa e^{-s} ds \\ &\leq \dfrac{\Gamma(\kappa+1)}{(m-2)^{\kappa+1}}.
\end{align*}
Putting things together, we have shown that
$$D(0) - D(t) \leq c_0 \int_{0}^t (m-1)  \dfrac{\Gamma(\kappa+1)}{(m-2)^{\kappa+1}} du \leq m c_0 \Gamma(\kappa+1) \dfrac{3^{\kappa+1}}{m^{\kappa+1}} t = C \dfrac{t}{m^\kappa}.$$
In the latter upper-bound, we used $m-2 \geq m/3$ since $m \geq 3$, and $C = c_0 \Gamma(\kappa+1) 3^{\kappa+1}$. As a consequence, $$\PP\left(W_{(2)} -W_{(1)}   > t \right) = m D(t) \geq m D(0) - m \dfrac{C}{m^\kappa} t =   1 - \dfrac{C}{m^{\kappa-1}} t.$$
Choosing $t = { \delta  \, m^{\kappa-1}} / {C}$, leads to the statement.
\end{proof}

The purpose of this lemma is to establish conditions ensuring the assumption of the previous lemma. 
\begin{lemma}\label{condition F vers kappa}  
     Let $W$ be a random variable on $\mathbb R_{+}$ with density $f_W$ and cumulative distribution function $F_W$. Suppose that there exists \( T_0 > 0 \) such that for all \( t \in (0,T_0) \), we have
\[
F_W(t) \geq c_1 t^d \quad \text{and} \quad f_W(t) \leq c_2 t^{d-1}.
\]  
Additionally, suppose that there exists $U>0$ such that, for all \( t \geq 0 \), we have \( f_W(t) \leq U \). Then,   for all \( t \geq 0 \), the following inequality holds
\[
f_W(t) \leq c F_W(t)^\kappa,
\]  
where $\kappa = 1 - 1/d$ and \( c = {c_1^{-\kappa}} \max\left\{c_2, {U} / {T_0^{d-1}}\right\} \).  
\end{lemma}  

\begin{proof}  
For all \( t \in (0, T_0) \), we have  
\[
f_W(t) \leq c_2 t^{d-1} \leq c_2 \left(\frac{F_W(t)}{c_1}\right)^{1-1/d} = \dfrac{c_2}{c_1^\kappa} F_W(t)^\kappa,
\]  
where \( \kappa = 1 - 1/d \).  For \( t \geq T_0 \), we have
\[
f_W(t) \leq \frac{U}{(c_1T_0^d)^\kappa} (c_1T_0^d)^\kappa \leq  \frac{U}{(c_1T_0^d)^\kappa} F_W(T_0)^\kappa \leq \frac{U}{(c_1T_0^d)^\kappa} F_W(t)^\kappa.
\]  
Setting  
\[
c = \max\left(\frac{c_2}{c_1^{\kappa}}, \frac{U}{(c_1T_0^d)^{\kappa}}\right) = \frac{1}{c_1^{\kappa}} \max\left(c_2, \frac{U}{T_0^{d-1}}\right),
\]  
we obtain, for all \( t \geq 0 \), 
$ f_W(t) \leq c F_W(t)^\kappa$, as desired. 
\end{proof}  

In the following lemma, we provide assumptions on $Z$ to obtain results on $W$ that will be useful for applying Lemma \ref{condition F vers kappa}.

\begin{lemma}\label{Z vers X}
 Let $ Z $ be a random variable in $\mathbb R^d$ with density $f_Z$. Let $x\in \mathbb R^d$ and $W =\| Z - x\| $. The following holds:
\begin{enumerate}[label={(\alph*)}]
    \item  If $f_{Z} $ is bounded  by $M > 0$ then, for almost all $\rho>0$, $f_W(\rho) \leq Md V_d \rho^{d-1}$. 
    \item If $f_{Z} $ is bounded  by $M>0$ with compact support included in $B (0, \rho_0)$, and $x \in B (0, \rho_0)$, then $f_W $ is bounded from above almost everywhere by $M d V_d 2^{d-1} \rho_0^{d-1}.$ 
    \item If $f_{Z}$ is bounded from below by $b > 0$ with support $S_{Z}$ and if there exists $c_d>0$ and $T_0>0$ such that $\lambda (S_{Z} \cap B(x, \tau ) ) \geq c_d \lambda   (B(x, \tau ))$ for all $\tau \in (0,T_0]$ and $x\in S_{Z}$, then for all $\rho \leq T_0, \, F_W(\rho) \geq c_d b V_d \rho^d$.
\end{enumerate}  
\end{lemma}  

\begin{proof}
Let us start by showing (a). Note that for any function $h: \mathbb R_{+} \to \mathbb R_{+}$, we have
\begin{equation*}
\mathbb   E [  h(\|{Z-x}\|)  ] = \int h(\|t-x\| )  f_{Z}(t)  dt \leq M  \int h(\|t-x\| ) dt = Md V_d \int h(\rho)\rho^{d-1} d\rho .     
\end{equation*}
Then almost everywhere $f_W(\rho) \leq Md V_d \rho^{d-1}$. Now we consider (b). 
From the first point, we have almost everywhere for $\rho > 0$, $f_W(\rho) \leq Md V_d \rho^{d-1}.$ Moreover, when $f_Z$ has compact support included in $B(0, \rho_0)$, we have that $W$ is supported on $[0,2\rho_0]$. We then have almost everywhere $f_W(\rho) \leq \sup_{\rho \leq 2\rho_0} Md V_d \rho^{d-1} = Md V_d 2^{d-1} \rho_0^{d-1}.$ 
To prove (c), note that we have, for all $\rho \leq T_0$,
\begin{align*}
 F_W(\rho) &= \PP({Z} \in B(x, \rho)) = \int_{S_{Z}\cap B (x,\rho)} f_{Z}(u) du \\ &\geq b  \lambda (S_{Z}\cap B (x,\rho) ) \geq b c_d \lambda (B(x,\rho) ) = bc_d V_d \rho^d.    
\end{align*} 
\end{proof}

\begin{lemma}\label{PV}
 Let $(Z_i)_{i=1,\dots,m} \subset  \mathbb R^d$. Let $\hat Z_i(z)$ be the $i$-th nearest neighbor of $z\in \mathbb R^d$ (breaking ties in favor of larger index). Let $ x\in \mathbb R^d$ and define $\mathcal V(x) = \{z\in\mathbb R^d\,:\, \hat Z_1(z) = \hat Z_1(x)\}$ and $ W_{(i)} = \| \hat Z_i(x) - x\| $, for $i=1,\ldots, m$. Let $P$ be a probability measure such that $ P ( B (x,t) )  \geq c_3 t^d $ for all $t\in (0,T_1)$ and $x \in S_{Z}$ and for some $T_1>0$. Then, whenever $ (W_{(2)} -  W_{(1)}) \leq 2T_1$, we have
   $$ P(\mathcal{V}(x) ) \geq \frac{c_3}{2^d}   (W_{(2)} -  W_{(1)}) ^d.$$
\end{lemma}

\begin{proof}
We have
\begin{align}\label{first_ineq}
     P  (\mathcal{V}(x) ) 
    & = \sum _{k=1} ^ m \ind_{V_k }(x) \,  P [ \mathcal{V}(x) ]  = \sum _{k=1} ^ m \ind_{V_k }(x) \,  P [  V_k ]
\end{align}
Remark that 
$$   B ( Z_k , \Delta_k /2) \subset  V_k      $$
where $\Delta _k  = \min _{i\neq k} \|Z_i - Z_k\|$. 
Hence
$$  P [  V_k ] \geq  P [ B ( Z_k , \Delta_k /2) ].$$
We have (second triangular inequality)
$$ \|Z_i- Z_k\| \geq \left|  \|Z_i- x\| - \|x- Z_k\| \right|, $$
but when $x\in V_k$, $\|x-Z_k\| \leq  \min_{i\neq k} \|Z_i- x\| $. Therefore
$$ \|Z_i- Z_k\| \geq   \|Z_i- x\| - \|x- Z_k\| \geq 0. $$
Hence, whenever $x\in V_k$,
$$ \Delta _k = \min_{i\neq k} \|Z_i- Z_k\| \geq   \min_{i\neq k} \|Z_i- x\| - \|x- Z_k\|.$$
but since $ W_i =\| Z_i - x\| $, for $i=1,\ldots, m$, and $W_{(i)} $ are the increasing ordered statistics, we have
$$\Delta _k \geq W_{(2)} - W_{(1)} := \Delta.$$
It follows that
$$  P  (V_k) \geq   P [ B ( Z_k , \Delta /2) ].   $$
Suppose that $(W_{(2)} -  W_{(1)})\leq 2T_1$,  using the assumption $ P ( B(x,t) )  \geq c_3 t^d $ for $t = \Delta/2 \in (0, T_1)$, we find 
$$ P  (V_k) \geq \frac{c_3}{2^d}  \Delta^d .$$ 
From \eqref{first_ineq}, it finally follows that,
$$  P (\mathcal{V}(x) ) \geq  \frac{c_3}{2^d} \Delta ^d \sum _{k=1} ^ m \ind_{V_k } (x) = \frac{c_3}{2^d} \Delta ^d.$$
\end{proof}

The following lemma is useful for proving Propositions \ref{centered tree not regular} and \ref{uniform tree not regular}.
\begin{lemma}\label{lemme moment}
   Let \( M, (M_i)_{i=1,\ldots, N} \) be a collection of independent and identically distributed random variables such that $\mathbb E[M] = 0$ and $\mathbb E[M^4] <\infty$. It holds \[\mathbb{E}\left[ \left( \sum_{i=1}^{N} M_i \right)^4\right] = N \mathbb{E}(M^4) + 3N (N-1) \mathbb{E}(M^2)^2. \]
\end{lemma}

\begin{proof}
We have \[   \left( \sum_{i=1}^{N} M_i \right)^4 = \sum_{i,p,q,r=1}^{N} M_i M_p M_q M_r.\]
Since the \( M_i \) are independent and centered, the expectation of each product \( M_i M_p M_q M_r \) will be zero if at least one of the indices is distinct. This restricts the analysis to cases where all indices are identical or two pairs of indices are identical. If all indices are identical, i.e. \( i = p = q = r \), then the expectation of \( M_i^4 \) contributes to the sum: $ \sum_{i=1}^{N} \mathbb{E}(M_i^4) = N \, \mathbb{E}(M^4)$. When two indices are identical and the other two are also identical, i.e. \( i = p \neq q = r \), we get a product of the form \( M_i^2 M_q^2 \). We have 3 choices either $i$ is equal to \( p \), \( q \), or \( r \). The remaining two indices must necessarily be equal. This yields: $ 3 \sum_{i \neq q} \mathbb{E}(M_i^2) \, \mathbb{E}(M_q^2) = 3N (N - 1) \, \mathbb{E}(M^2)^2.$ Combining the two terms, this proves the result.
\end{proof}

\end{document}